\input ./style/arxiv-general.cfg
\documentclass[MSNbibl,number,citesort,dvips]{arxbj}
\makeatletter
   \@ifpackageloaded{graphicx}{}{\usepackage{graphicx}}
\makeatother

\aid{0}
\volume{21}
\issue{3}
\pubyear{2015}
\firstpage{1467}
\lastpage{1493}
\doi{10.3150/14-BEJ611} 

\makeatletter

\newcommand{\RR}{\mathbb{R}}
\newcommand{\ZZ}{\mathbb{Z}}
\newcommand{\NN}{\mathbb{N}}
\newcommand{\cone}{\operatorname{cone}}
\newcommand{\PP}{\mathrm{P}}
\newcommand{\pa}{\operatorname{pa}}
\newcommand{\ch}{\operatorname{ch}}
\newcommand{\an}{\operatorname{an}}
\newcommand{\indep}{\mathop{\perp\!\!\!\perp}}
\newcommand{\indtri}[4]{#1 \indep #2 \, | \, #3 \, [#4] }
\newcommand{\mor}{\mathrm{mor}}
\newcommand{\clo}[1]{\overline{#1}}
\newcommand{\clod}[1]{\overline{#1'}}
\newcommand{\uedge}{\mathop{\,\rule[0.5ex]{0.8em}{0.1ex}}\,}

\newtheorem{them}{Theorem}[section]
\newtheorem{lem}[them]{Lemma}
\newtheorem{cor}[them]{Corollary}
\newtheorem{prop}[them]{Proposition}
\newremark{rem}[them]{Remark}
\newremark{exmp}[them]{Example}
\newproclaim{defn}[them]{Definition}

\newcommand{\eqref}[1]{(\ref{#1})}

\renewcommand{\emptyset}{\varnothing}

\makeatother

\begin{document}
\begin{frontmatter}

\title{Standard imsets for undirected and chain graphical models}
\runtitle{Imsets for undirected and chain graphical models}

\begin{aug}
\author[1]{\inits{T.}\fnms{Takuya}~\snm{Kashimura}\thanksref{1}} 
\and
\author[1,2]{\inits{A.}\fnms{Akimichi}~\snm{Takemura}\corref{}\thanksref{1,2}\ead[label=e2]{takemura@stat.t.u-tokyo.ac.jp}}
\address[1]{Department of
Mathematical Informatics,
Graduate School of Information Science and Technology,
University of Tokyo, 7-3-1 Hongo, Bunkyo-ku, Tokyo, 113-8656, Japan.\\ \printead{e2}}
\address[2]{JST, CREST, 5 Sanbancho, Chiyoda-ku, Tokyo, 102-0075, Japan}
\end{aug}

\received{\smonth{2} \syear{2011}}
\revised{\smonth{1} \syear{2014}}

%
\begin{abstract}
We derive standard imsets for undirected graphical models and
chain graphical models. Standard imsets for undirected graphical models
are described in terms of minimal triangulations for
maximal prime subgraphs of the undirected graphs. For describing
standard imsets
for chain graphical models, we first define a triangulation of
a chain graph. We then use the triangulation to generalize our results
for the undirected graphs to chain graphs.
\end{abstract}

%

\end{frontmatter}

\section{Introduction}
\label{sec:introduction}
The notion of imsets introduced by Studen\'y \cite{stu2005}
provides a very convenient algebraic method for encoding
all conditional independence (CI) models which hold under
a discrete probability distribution.
However, a class of imsets does not satisfy the uniqueness property:
a number of different imsets represent the same CI model.

Thus some questions related to the uniqueness property arise \cite
{stu2005}. One of them is
the task of characterizing equivalent imsets. For example, in the case
of classical graphical models \cite{lauritzen},
their equivalence classes are characterized by Andersson \textit{et al.} \cite
{andersson} and Frydenberg \cite{frydenberg}
in graphical terms.
Studen\'y \cite{stu1994} related a CI model induced by a imset to some
face of
a special polyhedral cone, and
an algorithm for CI inference based on this cone is studied in \cite{bou2010}.

Another question is to find a suitable representative for every
equivalence class.
This is motivated by a practical question about learning CI models (see
Section~4.4 in \cite{stu2001} and Section~4 in \cite{vomstu}).
As a subproblem of this, explicit expressions of imsets for important
classes of
graphical models, such as directed acyclic graphical (DAG) models and
decomposable models,
are given in \cite{stu2005}.
Imsets for some chain graphical (CG) models are also known \cite{sturovste}.
They are called standard imsets and have attractive simple forms.
One of their advantages is that they give a simple method to test whether
two graphs have the same CI model.
Another advantage is that it provides a translation of graphical models
into the framework of imsets.
Thus standard imsets offer a new algebraic approach for learning
graphical models \cite{stu2010,hem2010}.

In this paper, we derive standard imsets for undirected graphical (UG)
models and
general CG models. Our standard imsets generalize those for
DAG models and decomposable models. For UG models, we consider
all minimal triangulations of an undirected graph in accordance with
maximal prime
subgraphs and then use the standard imsets for minimal triangulations
(which are decomposable
models) for defining our standard imset. For CG models,
we first define a triangulation of a chain graph.
We then use the triangulation to generalize our results for undirected
graphs to chain graphs.

The organization of the paper is as follows.
In Section~\ref{sec:preliminaries}, we summarize basic definitions
and known facts on imsets and graphs, including standard imsets for
DAG models and decomposable models.
In Section~\ref{sec:UG}, we derive standard imsets for UG models.
In Section~\ref{sec:CG}, we introduce a notion of triangulation of a
chain graph and
based on the triangulation we derive standard imsets for CG models.
We conclude the paper with some remarks in Section~\ref{sec:remarks}.

\section{Preliminaries}
\label{sec:preliminaries}
In this section, we summarize our notation, definitions and relevant
preliminary results concerning conditional independence, imsets and
graphical models.

\subsection{Conditional independence and imsets}
\label{sec:imset}
First, we set up notation for conditional independence and
imsets following Studen\'y \cite{stu2005}.

Let $N$ be a finite set of variables
and let $\mathcal{P}(N) = \{ A\dvt  A \subseteq N \}$ denote the power
set of $N$.
For convenience, we write the union $A \cup B$ of subsets of $N$ as $AB$.
A singleton set $\{i\}$ is simply written as $i$.
As usual,
$\RR$, $\ZZ$, and $\NN$ denote
reals, integers and natural numbers, respectively.
For pairwise disjoint subsets, $A,B,C \subseteq N$,
we write this triplet by $\langle A, B \, | \, C \rangle$, and
the set of all disjoint triplets $\langle A, B \, | \, C \rangle$ over
$N$ by $\mathcal{T}(N)$.
As usual, for a probability distribution $\PP$ over $N$,
$\indtri{A}{B}{C}{\PP}$ denotes the conditional independence
statement of
variables in $A$ and in $B$ given the variables in $C$ under $\PP$.
The case $C = \emptyset$ corresponds to the marginal independence of
$A$ and $B$.
In this paper, we regard a triplet $\langle A, B \, | \, C \rangle$ as
an independence
statement.
Then the set of conditional independence statements under $\PP$ is
denoted as
\begin{eqnarray*}
\mathcal{M}_{\PP} = \bigl\{ \langle A, B \, | \, C \rangle \in\mathcal
{T}(N)\dvt  \indtri{A} {B} {C} {\PP} \bigr\}. %
\end{eqnarray*}
We call
$\mathcal{M}_{\PP}$
the \textit{conditional independence model} induced by $\PP$.

An \textit{imset} over $N$ is an integer-valued function $u \dvtx  \mathcal
{P}(N) \rightarrow\ZZ$, or alternatively,
an element of $\ZZ^{|\mathcal{P}(N)|}=\ZZ^{2^{|N|}}$.
The identifier $\delta_A$ of a set $A \subseteq N$ is defined as
\begin{eqnarray*}
\delta_A(B) &= %
\cases{ 1, &\quad  $B = A$,
\cr
0, & \quad $B \neq A$,
$B \subseteq N$. } %
\end{eqnarray*}
For a triplet $\langle A, B \, | \, C \rangle \in\mathcal{T}(N)$, a
\textit
{semi-elementary imset} $u_{\langle A, B \, | \, C \rangle}$ is
defined as
\begin{eqnarray*}
u_{\langle A, B \, | \, C \rangle} = \delta_{ABC} + \delta_C -
\delta_{AC} - \delta _{BC}. %
\end{eqnarray*}
If $A = a$ and $B = b$ are singletons, the imset $u_{\langle a, b \, |
\, C \rangle}$
is called \textit{elementary}.
The set of all elementary imsets is denoted by $\mathcal{E}(N)$.
Let $\cone(\mathcal{E}(N))\subseteq{\mathbb R}^{2^{|N|}}$ be
the polyhedral cone generated by all the elementary imsets. It can be
shown that
every elementary imset is a generator of an extreme ray of the $\cone
(\mathcal{E}(N))$
\cite{stu1994}.
A \textit{combinatorial imset} is an imset which can be written as a
non-negative integer combination of elementary imsets.
The set of all combinatorial imsets is denoted by $\mathcal{C}(N)$.
Let
\[
\mathcal{S}(N) = \cone\bigl(\mathcal{E}(N)\bigr) \cap\ZZ^{|\mathcal{P}(N)|}.
\]
An element of $\mathcal{S}(N)$ is called
a \textit{structural imset}. %
Note that $\mathcal{C}(N) \subseteq\mathcal{S}(N)$ by definition,
however, it is known that
this inclusion is strict for $|N| \ge5$ \cite{hemmecke}.

A conditional independence statement induced by a structural imset is
defined as follows:
%
\begin{defn}%
For $u \in\mathcal{S}(N)$ and a triplet $\langle A, B \, | \, C
\rangle \in\mathcal{T}(N)$,
we define a \textit{conditional independence statement with respect to
$u$} as
\begin{eqnarray*}
\indtri{A} {B} {C} {u} \quad \iff\quad \exists k \in\NN,\qquad  k \cdot u - u_{\langle A,
B \, | \, C \rangle} \in
\mathcal{S}(N). %
\end{eqnarray*}
The independence model induced by $u$ is denoted by
\[
\mathcal{M}_u = \bigl\{ \langle A, B \, | \, C \rangle \in
\mathcal{T}(N) \dvt  \indtri {A} {B} {C} {u} \bigr\}.
\]
\end{defn}
It can be shown that the structure of conditional independence models
induced by structural imsets
depends only on the face lattice of $\cone(\mathcal{E}(N))$, not on
each imset \cite{stu1994}.
Therefore implications of conditional independence models induced by imsets
correspond to those of faces of $\cone(\mathcal{E}(N))$.
The next lemma, which is very useful for our proofs in later sections,
follows from this fact.
%
\begin{lem}[(Studen\'y \cite{stu2005})] \label{lem:include}
For $u,u' \in\mathcal{S}(N)$,
\begin{eqnarray*}
\mathcal{M}_{u'} \subseteq\mathcal{M}_u \quad \iff\quad \exists k \in
\NN,\qquad  k \cdot u - u' \in\mathcal{S}(N).
\end{eqnarray*}
\end{lem}
The method of imsets is very powerful, because
conditional independence models induced by discrete probability
measures are
always represented by structural imsets.
%
\begin{them}[(Studen\'y \cite{stu2005})] \label{thm:main}
For every discrete probability measure $\PP$ over $N$,
there exists a structural imset $u \in\mathcal{S}(N)$ such that
$\mathcal{M}_u = \mathcal{M}_{\PP}$.
\end{them}

\subsection{Graphs and graphical models}
\label{sec:graphs}

Here we summarize relevant facts on graphs and graphical models following
Lauritzen \cite{lauritzen}, Studen\'y, Roverato and \v{S}t\v{e}p\'
{a}nov\'{a} %
\cite{sturovste},
Leimer \cite{leimer}, and Hara and Takemura \cite{hara}.

Throughout this paper, we consider a simple graph $G=(V(G), E(G))$, $V(G)=N$,
$E(G) \subseteq N \times N \setminus\{(a,a)\dvt  a\in N\}$.
An edge $(a,b) \in E(G)$ is \textit{undirected} if $(b,a) \in E(G)$.
We denote an undirected edge by
$a\uedge b$. If $(b,a) \notin E(G)$, we call $(a,b)$ \textit{directed}
and denote it by $a \to b$.
An \textit{undirected graph} (UG) contains only undirected edges,
while a directed graph
contains only directed ones.
The \textit{underlying graph} of a graph $G$ is the undirected graph
obtained from $G$ by replacing every directed edge with an undirected one.
For a subset $S \subseteq N$, $G_S$ denotes the subgraph of $G$ induced
by $S$.
In this paper when we refer to a subgraph of $G$, it is induced by some
subset of $N$.
A graph is \textit{complete} if all vertices are joined by an edge.
A subset $K \subseteq N$ is a \textit{clique} if $G_K$ is complete.
In particular, an empty set $K = \emptyset$ is a clique.
A clique $K$ is maximal if no proper superset $K' \supset K$ is a
clique in $G$.
$\mathcal{K}_G$ denotes the set of maximal cliques of $G$.

Two vertices $a,b\in N$ are adjacent if $(a,b)\in E(G)$ or $(b,a) \in E(G)$.
If $a \to b$, then $a$ is a \textit{parent} of $b$ and $b$ is a
\textit{child} of $a$.
For a vertex $c \in N$, we denote the set of parents and the set of
children of $c$ in $G$ by $\pa_G(c)$ and $\ch_G(c)$,
respectively.
For a subset $C \subseteq N$, let $\pa_G(C) = \bigcup_{c \in C} \pa
_G(c) \setminus C$ and
$\ch_G(C) = \bigcup_{c \in C} \ch_G(c) \setminus C$. %
We will omit the subscript $G$ if it is obvious from the context.

A \textit{path} of length $k \ge0$ from $a$ to $b$ is a sequence $a =
c_1, \dots, c_{k+1} = b$ of distinct vertices such that $(c_{i},
c_{i+1}) \in E(G)$ for $ i = 1, \dots, k$.
If a path contains only undirected edges, it is an undirected path and
otherwise (i.e., it contains
at least one directed edge) directed.
Note that some authors use the term ``semi-directed'' instead of ``directed''.
A vertex $a \in N$ is an \textit{ancestor} of $b \in N$ if there
exists a path from $a$ to $b$.
Let $\an_G(a)$ be the set of all ancestors of $a$.
The ancestral set $\an_G(C)$ of a subset $C \subseteq N$
is defined as $\an_G(C) = \bigcup_{c \in C} \an_G(c)$.
Note that $C \subseteq\an_G(C)$.
Let $c_1, \dots, c_k$ be a path with $(c_k, c_1) \in E(G)$.
Then we call the sequence $c_1, \dots, c_k, c_1$ a \textit{cycle} of
length $k$. Analogously to paths, a cycle is undirected if it contains
only undirected edges, otherwise directed.
A \textit{directed acyclic graph} (DAG) is a directed graph containing
no directed cycles.

A subset $C \subseteq N$ is said to be \textit{connected} if there
exists an undirected path from $a$ to $b$ for all $a, b \in C$ in the
subgraph $G_C$.
A \textit{connectivity component of $G$} is a maximal connected subset
in $G$ with respect to set inclusion.
The connectivity components in $G$ form a partition of $N$.
A~\textit{chain graph} (CG) $G$ is a graph %
whose connectivity components $C_1, \dots, C_m$ can be ordered such that
if $a \rightarrow b \in E(G)$ with $a \in C_i, b \in C_j$, then $i
< j$.
Equivalently, a chain graph is defined as a graph containing no
directed cycles.
The connectivity components of a chain graph are
called \textit{chain components}. The set of chain components of a
chain graph $G$ is denoted by $\mathcal{C}_G$.
The chain components are most easily found by removing all directed
edges from $G$
before taking connectivity components.
Both undirected graphs and directed acyclic graphs are chain graphs.
In fact, a chain graph is undirected provided $m=1$, and directed
acyclic if each chain component contains only one vertex.
Suppose two chain graphs $G,H$ have the same underlying graph.
Then we say $H$ is \textit{larger} than or equal to $G$
if $a \rightarrow b$ in $H$ implies $a \rightarrow b$ in $G$.
In this case, we write $H \ge G$.
From the definition, $H$ has more undirected edges than $G$ if $H$ is
larger than $G$.

We now discuss maximal prime subgraphs of an undirected graph $G$.
A non-empty subset $\emptyset\neq S \subset N$ is a \textit{separator}
if the set $N\setminus S$ is not connected.
$S = \emptyset$ is a separator if (and only if) $G$ is not connected.
A separator $S$ is a \textit{clique separator} if $S$ is a clique.
For two vertices $u,v \in N$ with $u\uedge v \notin E(G)$,
a set $S$ is called a \textit{$(u,v)$-separator} if $u$ and $v$ belong
to different components of $G_{N \setminus S}$.
A \textit{minimal vertex separator} is a minimal $(u,v)$-separator for
some $u,v \in N$ with respect to set inclusion %
relative to all $(u,v)$-separators. Note that a minimal vertex
separator for some $u,v$ maybe a strict
subset of a minimal vertex separator for another pair.
For $\langle A, B \, | \, C \rangle \in\mathcal{T}(N)$, we say that
$A$ and $B$ are
separated by $C$
if $C$ is $(a,b)$-separator for all $a \in A$ and $b \in B$.

A graph $G$ is \textit{prime} if $G$ has no clique separators.
Let $G_V$, $V \subseteq N$, be prime. Then $G_V$ is a \textit{maximal
prime subgraph} (mp-subgraph)
and $V$ is a \textit{maximal prime component} (mp-component) of $G$,
if %
there is no proper superset $V' \supset V$ such that $G_{V'}$ is prime.
From Lemma~2.1(iii) of \cite{leimer}, if $V_1$ and $V_2$ are distinct
prime components then $G_{V_1 \cap V_2}$ is complete.
The set of mp-components of $G$ is denoted by $\mathcal{V}_G$.
There exists an order $V_1, \dots, V_m, m = |\mathcal{V}_G|$, of
$\mathcal{V}_G$
such that
\begin{eqnarray*}
\forall i \in\{ 2,\dots,m \}, \exists k \in\{ 1, \dots, i-1 \},\qquad  S_i
\equiv V_i \cap\bigcup_{j < i}
V_j \subseteq V_k. %
\end{eqnarray*}
This sequence is said to be \textit{D-ordered}, or alternatively, to
have a \textit{running intersection property} (RIP) \cite{lauritzen}.
For each $i$, $S_i$ is a clique minimal vertex separator.
An important fact about RIP is that for each $i$, $\bigcup_{j < i} V_j
\setminus S_i$ and $V_i \setminus S_i$ are separated by $S_i$
in $H_{V_1 \cup\cdots\cup V_i}$ by applying Corollary~2.7(i) of \cite
{leimer} recursively.
Define $\mathcal{S}_G = \{ S_2, \dots, S_m \}$.
Then $\mathcal{S}_G$ is the set of all clique minimal vertex
separators in $G$.
Moreover, the number of $S \in\mathcal{S}_G$ which appears among
$S_2, \dots, S_m$ may be more than one.
This number is called the \textit{multiplicity} of $S$ in $G$, and
written as $\nu_G(S)$.
For any undirected graph $G$, $\mathcal{V}_G, \mathcal{S}_G$ and
$\{ \nu_G(S) \}_{S \in\mathcal{S}_G}$ are uniquely defined \cite{leimer}.

In graphical models, the class of models induced by \textit
{decomposable graphs} are well studied,
because it has many good properties.
There are several equivalent definitions of decomposable graphs.
One of them is based on the decomposability of graphs.
For an undirected graph $G$ and a triplet $\langle A, B \, | \, C
\rangle$ with $N=A
\cup B \cup C$,
we say that \textit{$\langle A, B \, | \, C \rangle$ decomposes $G$
into the subgraphs
$G_{AC}$ and $G_{BC}$}
if $C$ is a clique and separates $A$ and $B$.
The decomposition is proper if $A, B \neq \emptyset$.
An undirected graph $G$ is decomposable if it is complete or there
exists $\langle A, B \, | \, C \rangle$
which properly decomposes $G$ into decomposable subgraphs $G_{AC}$ and $G_{BC}$.
Decomposable graphs are characterized in terms of mp-subgraphs by
Leimer \cite{leimer}.
An undirected graph $G$ is decomposable if and only if all
mp-components of $G$ are cliques.
Furthermore, for every undirected graph $G$ with
mp-components $V_1, \dots, V_m \in\mathcal{V}_G$,
there exists a decomposable graph $G'$ such that $V_1, \dots, V_m$ are
maximal cliques of $G'$.
The graph $G'$ is obtained by adding edges in such a way that $V_1,
\dots, V_m$ are cliques.

Another equivalent definition is a \textit{chordal graph}, or
alternatively \textit{triangulated graph}.
An undirected graph is chordal if every cycle of length more than or
equal to four has a chord, that is, an edge between two non-consecutive
vertices of the cycle.
An undirected graph is chordal if and only of it is decomposable \cite
{lauritzen}.

\subsection{Conditional independence models induced by graphs}
Here we summarize known facts on conditional independence models
induced by graphs.

For directed acyclic graphs,
there are two equivalent separation criteria \textit{d-separation}
\cite{pearl,verma} and \textit{moralization} \cite{laudawlarlei}.
However we omit their details because we do not need them in this paper.
For a triplet $\langle A, B \, | \, C \rangle \in\mathcal{T}(N)$, we
write $\indtri
{A}{B}{C}{G}$ if
$A$ and $B$ are separated given $C$ by these criteria.
Every directed acyclic graph $G$ induces the formal independence model
%
\begin{eqnarray}\label{eq:mg}
\mathcal{M}_G = \bigl\{ \langle A, B \, | \, C \rangle \in
\mathcal{T}(N) \dvt  \indtri {A} {B} {C} {G} \bigr\},
\end{eqnarray}
which we call a \textit{DAG model}.
A probability measure $\PP$ over $N$ is \textit{Markovian} with
respect to a directed acyclic graph $G$
if $\mathcal{M}_G \subseteq\mathcal{M}_{\PP}$ and \textit
{perfectly Markovian} if the converse inclusion also holds.

For an undirected graph $G$ and $\langle A, B \, | \, C \rangle \in
\mathcal{T}(N)$,
we have $\indtri{A}{B}{C}{G}$
if $A$ and $B$ are separated by $C$ in $G$ \cite{lauritzen,pearl}. %
An \textit{UG model} $\mathcal{M}_G$ is again defined by \eqref{eq:mg}.
The definitions of a Markovian and a perfectly Markovian measure are analogous
to the case of DAG models.
It is known that a perfectly Markovian discrete measure exists for
every undirected graph \cite{geiger1993}.
A~decomposable model is defined as an independence model induced by a
decomposable graph.
A~decomposable model is simultaneously an UG model and a DAG model.

Finally, we discuss chain graphs.
A popular separation criterion for chain graphs is \textit
{moralization} \cite{frydenberg}.
For a chain graph $G$ and a triplet $\langle A, B \, | \, C \rangle
\in\mathcal
{T}(N)$, let $H = G_{\an(ABC)}$.
A \textit{moral graph} $H^{\mor}$ of $H$ is the undirected graph
obtained by adding an undirected edge $a\uedge b$ to the underlying
graph of $H$
whenever there is a chain component $C' \in\mathcal{C}_H$ such that
$a,b \in\pa(C')$ and $a$ and $b$ are not adjacent in $H$.
We define $\indtri{A}{B}{C}{G}$
if $\indtri{A}{B}{C}{H^{\mor}}$ holds.
The definitions of a \textit{CG model}, a Markovian measure and a
perfectly Markovian measure are analogous to the other graphs.
It is known that a perfectly Markovian discrete measure exists for
every chain graph \cite{stubou}.

An important concept about chain graphs is the equivalence for graphs
\cite{stu2005}.
We say that $G$ and $H$ are \textit{equivalent} if $\mathcal{M}_G =
\mathcal{M}_H$.
Equivalent chain graphs are characterized by Frydenberg \cite{frydenberg}.
A \textit{complex} in $G$ is a subgraph of $G$ of the form $c_0 \to
c_1\uedge\cdots\uedge c_k \leftarrow c_{k+1}, k \ge1$,
and no other edges between $c_0, c_1, \dots, c_{k+1}$ exist in $G$.
%
\begin{them}[(Frydenberg \cite{frydenberg})] \label{thm:cg_equivalent}
Two chain graphs are %
equivalent if and only if their underlying graphs coincide
and they have the same complexes.
\end{them}
A more important fact is that every %
equivalence class has one distinguished representative.
%
\begin{them}[(Frydenberg \cite{frydenberg})] \label{thm:frydenberg}
Every %
equivalence class $\mathcal{H}$ of chain graphs %
has the largest element $H_{\infty} \in\mathcal{H}$ such that $H \le
H_{\infty}$
for all $H \in\mathcal{H}$.
\end{them}

\subsection{Standard imsets for directed acyclic graphs and
decomposable graphs}
\label{sec:standard}
Let $G$ be a directed acyclic graph.
A \textit{standard imset for $G$} is defined as follows \cite{stu2005}:\vspace*{1pt}
%
\begin{eqnarray}\label{eq:dagim}
u_G = \delta_N - \delta_\emptyset+ \sum
_{i \in N} \{ \delta_{\pa
(i)} - \delta_{\{ i \} \cup\pa(i)} \}.
\end{eqnarray}
This standard imset %
is a unique representative for equivalent graphs.
%
\begin{lem}[(Studen\'y \cite{stu2005})] \label{lem:dagim}
Let $G$ be a directed acyclic graph. %
Then $u_G \in\mathcal{C}(N)$ and $\mathcal{M}_G = \mathcal
{M}_{u_G}$ hold.
Moreover, for a directed acyclic graph $G'$, $\mathcal{M}_{G} =
\mathcal{M}_{G'}$ if and only if $u_{G} = u_{G'} $.
\end{lem}

A standard imset for a decomposable graph $H$ is defined by the sets of
maximal cliques and clique minimal vertex separators in $H$ \cite{stu2005}:
%
\begin{eqnarray}\label{eq:dgim}
u_H = \delta_N - \sum_{K \in\mathcal{K}_H}
\delta_{K} + \sum_{S
\in\mathcal{S}_H}
\nu_H(S) \cdot\delta_S.
\end{eqnarray}

%
\begin{figure}[b]

\includegraphics{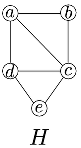}

\caption{A decomposable graph $H$.}
\label{fig:chordal}
\end{figure}

%
\begin{exmp}
Put $N = \{ a,b,c,d,e \}$ and consider the decomposable graph $H$
shown in Figure~\ref{fig:chordal}.
The sets of maximal cliques and clique minimal vertex separators in $H$
are $\mathcal{K}_H = \{ abc, acd, cde \}$
and $\mathcal{S}_H = \{ ac, cd \}$ (with multiplicities $\nu_H(ac)
= \nu_H(cd) = 1$).
Then the standard imset for $H$ is
\begin{eqnarray*}
u_H &=& \delta_{abcde} - \delta_{abc} -
\delta_{acd} -\delta_{cde} + \delta_{ac} +
\delta_{cd}
\\
&=& u_{\langle b, e \, | \, acd \rangle} + u_{\langle a, e \, | \, cd
\rangle} + u_{\langle b, d \, | \, ac \rangle}.
\end{eqnarray*}
For a complete graph, its standard imset is the zero imset.
\end{exmp}

Since decomposable models can be viewed as DAG models, their imsets
\eqref{eq:dagim} and \eqref{eq:dgim} lead to the same imset.
%
\begin{lem}[(Studen\'y \cite{stu2005})]
For every decomposable graph $H$, there exists a directed acyclic graph $G$
such that $\mathcal{M}_G = \mathcal{M}_H$ and $u_G = u_H$.
\end{lem}
This implies that for a decomposable graph $H$, we have
$u_H \in\mathcal{C}(N)$ and $\mathcal{M}_H = \mathcal{M}_{u_H}$
from Lemma~\ref{lem:dagim}.

As discussed in Section~\ref{sec:introduction}, these imsets for directed
acyclic and decomposable graphs are not the only combinatorial ones
representing their graphical models.
However they are the simplest, ``standard'' representations \cite{bou2010}.
A standard imset gives a simpler criterion of testing a conditional
independence statement than other imsets.
%
\begin{lem}[(Bouckaert \textit{et al.} \cite{bou2010})] \label{lem:criteria}
For a directed acyclic (resp. decomposable) graph $G$, $\langle A, B \,
| \, C \rangle
\in\mathcal{M}_G$ if and only if $u_G - u_{\langle A, B \, | \, C
\rangle} \in
\mathcal{C}(N)$,
which is also equivalent to $u_G - u_{\langle A, B \, | \, C \rangle}
\in\mathcal
{S}(N)$, where $u_G$ is the standard imset in \eqref{eq:dagim} or
\eqref{eq:dgim}.
\end{lem}

\section{Standard imsets for general undirected graphs}
\label{sec:UG}
In this section, we derive imsets for general undirected graphs. %
Our construction is based on a concept of a triangulation.

\subsection{General undirected graphical models}
For generalizing the result of decomposable graphs to general
undirected graphs,
consider constructing a decomposable graph from a given undirected
graph by adding edges.
The resulting graph is called a \textit{triangulation} of the input
graph \cite{heggernes}.
A triangulation $G'$ of $G$ is minimal if there is no triangulation
$G''$ of $G$
such that $E(G'') \subset E(G')$.
From Lemma~2.21 of \cite{lauritzen}, it follows that $G'$ is a minimal
triangulation of $G$ if and only if
removing any edge in $E(G')\setminus E(G)$ from $G'$ makes the
resulting graph non-decomposable.
In general, there are many minimal triangulations of a graph.
In the following, we denote the set of
all minimal triangulations of $G$ by $\mathfrak{T}(G)$.
As for separations of an input graph and a minimal triangulation, the
following lemma holds.
%
\begin{lem} \label{lem:ce}
For every undirected graph $H$ and a triplet $\langle A, B \, | \, C
\rangle \in
\mathcal{M}_H$,
there exists a minimal triangulation $H'$ of $H$ such that $\langle A,
B \, | \, C \rangle \in\mathcal{M}_{H'}$.
\end{lem}
\begin{pf}
It suffices to show the existence of a triangulation $H'$ such that
$\langle A, B \, | \, C \rangle \in\mathcal{M}_{H'}$.
In fact, if $H'$ is not minimal, we can obtain a minimal triangulation
by removing edges from $H'$, because removing edges does not destroy
the relation $A \indep B \,|\, C$.

We construct a desired triangulation as follows (see Figure~\ref{fig:ug_construction}).
Let $N = A' \cup B' \cup C'$ be a partition of the vertex set such that
\begin{eqnarray*}
A' &=& \{ i \in N \dvt  i \mbox{ is connected with } A \mbox{ in }
H_{N
\setminus C} \},
\\
C' &=& C \quad \mbox{and} \quad B' = N \setminus
A'C'.
\end{eqnarray*}
Construct the graph $H'$ by adding edges so that $H'_{A'C'}$ and
$H'_{B'C'}$ are complete.
This $H'$ is clearly decomposable, and hence, a triangulation of $H$.
From the construction, $A'$ and $B'$ are not connected to each other in
$H'_{N \setminus C'}$.
Thus $A \subseteq A'$ and $B \subseteq B'$ are not connected to each
other in $H'_{N \setminus C}$,
which means $\langle A, B \, | \, C \rangle \in\mathcal{M}_{H'}$.
\end{pf}
%
\begin{figure}

\includegraphics{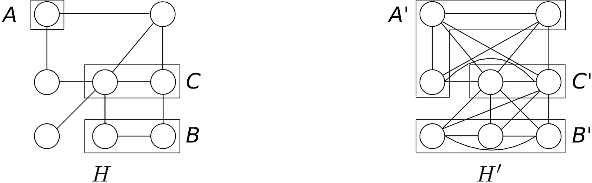}

\caption{A construction of $H'$ from an undirected graph $H$ in
Lemma \protect\ref{lem:ce}.}
\label{fig:ug_construction}
\end{figure}

For a general undirected graph $H$, we can obtain an imset representing
this UG model by using all minimal triangulations.
The following theorem is the first main result of this paper.
%
\begin{them} \label{thm:ug}
Let $H$ be an undirected graph.
Put
%
\begin{eqnarray}\label{eq:ugim1}
v_H = \sum_{H' \in\mathfrak{T}(H)} u_{H'},
\end{eqnarray}
where $\mathfrak{T}(H)$ is the set of minimal triangulations of $H$ and
$u_{H'}$ for $H' \in\mathfrak{T}(H)$ are defined by \eqref{eq:dgim}.
Then $v_H \in\mathcal{C}(N)$ and $\mathcal{M}_H = \mathcal{M}_{v_H}$.
\end{them}
\begin{pf}
Since the class of combinatorial imsets is closed under the addition,
it is evident that the imset $v_H$ is combinatorial.

For every undirected graph $H$,
there exists a discrete probability measure $\PP$ with $\mathcal
{M}_{\PP} = \mathcal{M}_H$ \cite{geiger1993}.
Moreover, Theorem~\ref{thm:main} implies that there is a structural
imset $w
\in\mathcal{S}(N)$ such that $\mathcal{M}_{\PP} = \mathcal{M}_w$.
Then, for $H' \in\mathfrak{T}(H)$, we have
\[
\mathcal{M}_{u_{H'}} = \mathcal{M}_{H'} \subseteq
\mathcal{M}_H = \mathcal{M}_{\PP} = \mathcal{M}_w,
\]
which implies $k_{H'} \cdot w - u_{H'} \in\mathcal{S}(N)$ for some
$k_{H'} \in\NN$ from Lemma~\ref{lem:include}.
Therefore, putting $k = \sum_{H' \in\mathfrak{T}(H)} k_{H'}$, it
follows that $k \cdot w - v_H \in\mathcal{S}(N)$.
That is, $\mathcal{M}_{v_H} \subseteq\mathcal{M}_w = \mathcal{M}_H$.

Conversely, for every $\langle A, B \, | \, C \rangle \in\mathcal
{M}_H$, there exists
$H' \in\mathfrak{T}(H)$ such that $\langle A, B \, | \, C \rangle
\in\mathcal{M}_{H'}$
from Lemma~\ref{lem:ce}. Thus, $u_{H'} - u_{\langle A, B \, | \, C
\rangle} \in\mathcal
{S}(N)$ from Lemma~\ref{lem:criteria}.
Hence, we have
\[
v_H - u_{\langle A, B \, | \, C \rangle} = \sum_{H'' \in\mathfrak
{T}(H) \setminus
H'}
u_{H''} + (u_{H'} - u_{\langle A, B \, | \, C \rangle}) \in \mathcal{S}(N),
\]
which implies $\langle A, B \, | \, C \rangle \in\mathcal{M}_{v_H}$.
\end{pf}

The imset $v_H$ in \eqref{eq:ugim1} is a generalization of the case of
decomposable graphs, because for a decomposable graph $H$, the set of
minimal triangulations contains $H$ only.
An example of this imset is given in the next section.

\subsection{Some consideration toward a definition of standard imsets
for general undirected graphs}
The imset defined in the last section through all minimal
triangulations has `extra' additional parts as shown
in the following example.
%
\begin{exmp} \label{exmp:ug_imset}
Put $N = \{ a,b,c,d,e \}$. Consider the graph $H$ in Figure~\ref{fig:H} and
its minimal triangulations $H_1,H_2$.
%
\begin{figure}

\includegraphics{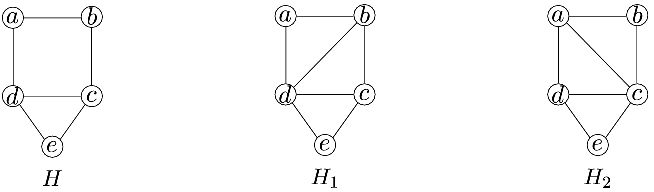}

\caption{Non-decomposable graph $H$ and its minimal triangulations
$H_1, H_2$.}
\label{fig:H}
\end{figure}
Then the imset $v_H$ in \eqref{eq:ugim1} is
\begin{eqnarray*}
v_H &= &u_{H_1} + u_{H_2}
\\
&=& (\delta_N - \delta_{abd} - \delta_{bcd} -
\delta_{cde} + \delta _{bd} + \delta_{cd})
\\
&&{} + (\delta_N - \delta_{abc} - \delta_{acd}-
\delta_{cde} + \delta _{ac} + \delta_{cd})
\\
& =& (u_{\langle ab, e \, | \, cd \rangle} + u_{\langle a, c \, | \, bd
\rangle}) + (u_{\langle ab, e \, | \, cd \rangle} +
u_{\langle b, d
\, | \, ac \rangle})
\\
&= &2 \cdot u_{\langle ab, e \, | \, cd \rangle} + u_{\langle a, c \, |
\, bd \rangle} + u_{\langle b, d \, | \, ac \rangle}.
\end{eqnarray*}
It can be seen that $ab \indep e\, |\, cd$ holds in both $H_1$ and $H_2$.
This is expressed as the coefficient 2 of $u_{\langle ab, e \, | \, cd
\rangle}$.
Now consider an imset $u_H$ with this coefficient 1, that is,
%
\begin{eqnarray}\label{eq:new}
u_H =u_{\langle ab, e \, | \, cd \rangle} + u_{\langle a, c \, | \, bd
\rangle} + u_{\langle b, d \, | \, ac \rangle}.
\end{eqnarray}
From Lemma~\ref{lem:ce}, $\langle A, B \, | \, C \rangle \in\mathcal
{M}_H$ is equivalent to
$\langle A, B \, | \, C \rangle \in\mathcal{M}_{H'}$ for some
minimal triangulation
$H'$ of $H$.
Hence, for example, letting $\langle A, B \, | \, C \rangle \in
\mathcal{M}_{H_1}$, we have
\begin{eqnarray*}
&&u_{H_1} - u_{\langle A, B \, | \, C \rangle} \in\mathcal{S}(N)
\\
&&\quad \Longrightarrow\quad  u_{H_1} + u_{\langle b, d \, | \, ac \rangle} - u_{\langle A, B \, | \, C \rangle} \in
\mathcal{S}(N)
\\
&&\quad \Longleftrightarrow\quad  u_H - u_{\langle A, B \, | \, C \rangle} \in \mathcal{S}(N)
\\
&&\quad \Longrightarrow\quad \langle A, B \, | \, C \rangle \in\mathcal{M}_{u_H}.
\end{eqnarray*}
Since the same result holds for $\langle A, B \, | \, C \rangle \in
\mathcal
{M}_{H_2}$, we have $\mathcal{M}_{v_H} = \mathcal{M}_H \subseteq
\mathcal{M}_{u_H}$.
Also, since $v_H - u_H = u_{\langle ab, e \, | \, cd \rangle} \in
\mathcal{S}(N)$, we
have $\mathcal{M}_{v_H} \supseteq\mathcal{M}_{u_H}$ from Lemma~\ref
{lem:include}.
Thus $\mathcal{M}_{v_H} = \mathcal{M}_{u_H} = \mathcal{M}_H$.
\end{exmp}

Note that a graph such as the one in Figure~\ref{fig:exponential} has an
exponential number of minimal triangulations,
which makes infeasible to calculate $v_H$ in \eqref{eq:ugim1} actually.
%
\begin{figure}

\includegraphics{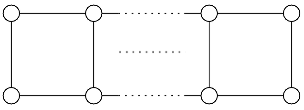}

\caption{A graph with an exponential number of minimal triangulations.}
\label{fig:exponential}
\end{figure}
%
\begin{figure}[b]

\includegraphics{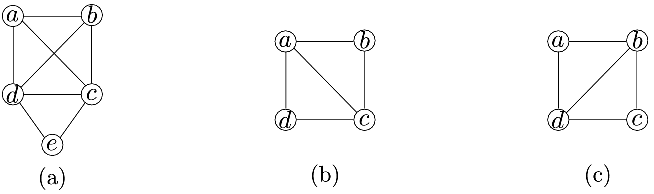}

\caption{Mp-subgraphs in the new imset.}
\label{fig:new_imset}
\end{figure}

The above examples suggest that it suffices to use only minimal
triangulations of each mp-subgraph and
not of the whole of the graph.
In particular, the new imset \eqref{eq:new} in Example~\ref{exmp:ug_imset}
seems to be defined as follows:
First, consider the graph obtained by adding edges to the input graph
in such a way that all mp-subgraphs are complete (Figure~\ref{fig:new_imset}(a)),
and consider its standard imset ($u_{\langle ab, e \, | \, cd \rangle
}$). Next, for
each mp-subgraph which is not complete,
consider their minimal triangulations (Figure~\ref{fig:new_imset}(b), (c))
and their standard imsets ($u_{\langle a, c \, | \, bd \rangle},
u_{\langle b, d \, | \, ac \rangle}$).
We show in the following sections that this idea is correct.

\subsection{Minimal triangulations and mp-subgraphs}
We show in this section that all minimal triangulations for an
undirected graph are obtained
by computing minimal triangulations for each mp-subgraph.

The following facts give the way of adding edges to obtain a minimal
triangulation:
%
\begin{lem}[(Ohtsuki \textit{et al.} \cite{ohtsuki})] \label{lem:ohtsuki}
A triangulation $H'$ of an undirected graph $H$ is minimal if and only if
for each $u\uedge v$ added by this triangulation,
no $(u,v)$-separators of $H$ is a clique in $H'$.
\end{lem}
%
\begin{lem}
\label{lem:mp-separation-obvious}
For an ordering $V_1, \dots, V_m$ of mp-components of an undirected
graph $H$ satisfying RIP,
$\bigcup_{k < i} V_k \setminus S_i$ and $V_i \setminus S_i$ are
separated by $S_i$ in the whole graph $H$ for each $i$.
\end{lem}
\begin{pf}
From Corollary~2.7(i) of \cite{leimer},
$\bigcup_{k < i} V_k \setminus S_i$ and $V_i \setminus S_i$ are
separated by $S_i$ in $H_{V_1 \cup\cdots\cup V_i}$ for each $i$.
For $i = m$, the desired conclusion already holds. Thus we show in the
case of $i < m$.
Suppose that there exists $p>i$ such that there exists a path from
some vertex $a \in\bigcup_{k < i} V_k \setminus S_i$ to some vertex
$b \in V_i \setminus S_i$ in $H_{V_1 \cup\cdots\cup V_p \setminus S_i}$.
Let $p > i$ be the minimum number with this property.
Choose a vertex $x \in V_p \setminus S_p$ of the path.
Since $\bigcup_{k < p} V_k \setminus S_p$ and $V_p \setminus S_p$ are
separated by $S_p$ in $H_{V_1 \cup\cdots\cup V_p}$
and $x$ leads to $a, b \in\bigcup_{k < p} V_k$,
the path must contain vertices of $S_p$.
Since $S_p$ is a clique and is contained in $V_q$ for some $1 \le q <
p$ from the definition of RIP,
there also exists a path from $a$ to $b$ in $H_{V_1 \cup\cdots\cup
V_{p-1} \setminus S_i}$.
However this contradicts minimality of $p$, and hence the desired
conclusion holds.
\end{pf}

From these lemmas, we have the following result about the relation
between mp-subgraphs and minimal triangulations of a graph.
%
\begin{lem} \label{lem:tri-mp}
For an undirected graph $H$, a graph $H'$ obtained by a minimal
triangulation of each mp-subgraph is a minimal triangulation of $H$.
Conversely, all minimal triangulations of $H$ are obtained in this way.
\end{lem}
\begin{pf}
Let $w$ be a cycle $a_1, \dots, a_n, a_{n+1} = a_1, n \ge4$, of
length more than or equal to 4 in $H'$.
First, consider the case that $w$ is not contained in one mp-component.
Let $V_i$ be the last mp-component in an ordering $V_1, \dots, V_m$
satisfying RIP such that $w$ intersects
$V_i\setminus S_i$. Choose a vertex $x$ of $w$ such that $x \in V_i
\setminus S_i$.
Since no edge in $H'$ outside mp-components is added, two (distinct)
branches of $w$ out of
$x$ lead to distinct elements $y, z \in S_i$. Since $S_i$ is a clique,
$y$ and $z$ are adjacent in $H$,
which means that the cycle $w$ has a chord.
We next consider the case $\{ a_1, \dots, a_n \} \subseteq V$ for
some mp-component $V \in\mathcal{V}_{H}$ in $H$.
Since $H'_{V}$ is decomposable, the cycle $w$ also has a chord in $H'$.
Therefore $H'$ is decomposable.
Moreover,
from an equivalent characterization of a minimal triangulation, it follows
that $H'$ is minimal since removing one edge
from $H'$ makes it non-decomposable.
Thus the first statement is proved.

To prove the converse,
consider
an edge $u \uedge v$ added by a triangulation $H'$ of $H$.
Let $u \in V_i \setminus S_j, v \in V_j \setminus S_j$ and $i < j$.
Then by Lemma~\ref{lem:mp-separation-obvious}
$u \in\bigcup_{k < j} V_k \setminus S_j$ and $v \in V_j \setminus
S_j$ are separated by $S_j$ in $H$.
Since $S_j$ is also a clique in $H'$, $H'$ is not a minimal
triangulation from Lemma~\ref{lem:ohtsuki}.
\end{pf}

\subsection{Definition and properties of standard imsets for
undirected graphs}
We define a standard imset for an undirected graph using Lemma~\ref
{lem:tri-mp}.
%
\begin{defn} \label{defn:imset-for-undirected-graph}
For an undirected graph $H$, a standard imset $u_H$ for $H$ is defined as
%
\begin{eqnarray}\label{eq:ugim}
u_H = \delta_N - \sum_{V \in\mathcal{V}_{H}}
\delta_{V} + \sum_{S
\in\mathcal{S}_H}
\nu_H(S) \cdot\delta_S + \sum
_{V \in\mathcal
{V}_{H}} \sum_{G \in\mathfrak{T}(H_V)}
u_{G},
\end{eqnarray}
where for each $G \in\mathfrak{T}(H_V), V \in\mathcal{V}_{H}$,
$u_G$ is the standard imset given by \eqref{eq:dgim}:
\[
u_{G} = \delta_{V} - \sum_{K \in\mathcal{K}_G}
\delta_K + \sum_{S
\in\mathcal{S}_G}
\nu_G(S) \cdot\delta_S.
\]
\end{defn}
Note that, if $H$ is decomposable, the last term of $u_H$ vanishes
because all mp-components are cliques \cite{leimer}.
Thus this imset coincides with \eqref{eq:dgim}.

We show that this imset represents an UG model.
%
\begin{them} \label{thm:ug_main}
For an undirected graph $H$, define $u_H$ as \eqref{eq:ugim}.
Then $u_H \in\mathcal{C}(N)$ and $\mathcal{M}_H = \mathcal{M}_{u_H}$.
\end{them}
\begin{pf}
The first three terms of \eqref{eq:ugim} correspond to the standard
imset for the decomposable graph such that all $V \in\mathcal{V}_H$
are cliques.
Thus, this imset is combinatorial, and hence, $u_H \in\mathcal{C}(N)$.

Let $H'$ be a minimal triangulation of $H$. Since a minimal
triangulation is done in each mp-subgraph from Lemma~\ref{lem:tri-mp},
the following relations hold:
\begin{eqnarray*}
\mathcal{K}_{H'} &=& \bigcup_{V \in\mathcal{V}_{H}}
\mathcal{K}_{H'_V},\qquad  \mathcal{S}_{H'} = \mathcal{S}_H
\cup \biggl( \bigcup_{V \in
\mathcal{V}_{H}} \mathcal{S}_{H'_V}
\biggr),
\\
\mathcal{K}_{H'_{V_1}} \cap\mathcal{K}_{H'_{V_2}} &=& \emptyset,\qquad
\mathcal{S}_{H'_{V_1}} \cap\mathcal{S}_{H'_{V_2}} = \emptyset, \qquad \forall
V_1, V_2 \in\mathcal{V}_{H},
V_1 \neq V_2,
\\
\mathcal{S}_H \cap\mathcal{S}_{H'_V} &= &\emptyset, \qquad \forall
V \in \mathcal{V}_{H},
\qquad
\nu_{H}(S) = \nu_{H'}(S),\qquad  \forall S \in
\mathcal{S}_H.
\end{eqnarray*}
To verify the disjointness $\mathcal{S}_H \cap\mathcal{S}_{H'_V} =
\emptyset$,
we note the fact that for every mp-component $V$, the elements of
$\mathcal{S}_{H'_V}$
are not cliques in $H_V$, because otherwise
$S \in\mathcal{S}_{H'_V}$, being a separator in $H'_V$, is a clique
separator in $H_V$,
which contradicts the primeness of $H_V$.
Now %
consider RIP ordering $V_1, \dots, V_m$ of mp-components of $H$.
For the last mp-component $V_m$,
the elements of $\mathcal{S}_{H'_{V_m}}$ are not cliques in $H$ and intersect
$V_m \setminus S_m$.
Since elements of $\mathcal{S}_H$ are cliques in $H$
and since the elements of $\mathcal{S}_{H'_{V_i}}$
for $i < m$ do not intersect $V_m \setminus S_m$, the class $\mathcal
{S}_{H'_{V_m}}$
is disjoint with those other ones. By decreasing induction on $m$,
the disjointness $\mathcal{S}_H \cap\mathcal{S}_{H'_V} = \emptyset$ holds.

Hence a standard imset for the decomposable graph $H'$ given by \eqref
{eq:dgim} is\vspace*{1pt}
%
\begin{eqnarray}\label{eq:u_H'}
u_{H'} &=& \delta_N - \sum_{K \in\mathcal{K}_{H'}}
\delta_K + \sum_{S \in\mathcal{S}_{H'}}
\nu_{H'}(S) \cdot\delta_S
\nonumber
\\
&=& \delta_N - \sum_{V \in\mathcal{V}_{H}} \sum
_{K \in\mathcal
{K}_{H'_V}} \delta_K + \sum
_{S \in\mathcal{S}_H} \nu_{H}(S) \cdot \delta_S + \sum
_{V \in\mathcal{V}_{H}} \sum_{S \in\mathcal{S}_{H'_V}} \nu
_{H'}(S) \cdot\delta_S
\nonumber
\\
&=& \delta_N - \sum_{V \in\mathcal{V}_{H}}
\delta_V + \sum_{S \in
\mathcal{S}_H}
\nu_{H}(S) \cdot\delta_S
\\
&&{} + \sum_{V \in\mathcal{V}_{H}} \biggl\{ \delta_V -
\sum_{K \in
\mathcal{K}_{H'_V}} \delta_K + \sum
_{S \in\mathcal{S}_{H'_V}} \nu _{H'_V}(S) \cdot\delta_S
\biggr\}
\nonumber
\\
&=& \delta_N - \sum_{V \in\mathcal{V}_{H}}
\delta_V + \sum_{S \in
\mathcal{S}_H}
\nu_{H}(S) \cdot\delta_S + \sum
_{V \in\mathcal
{V}_{H}} u_{H'_V}.\nonumber
\end{eqnarray}
In particular, the comparison with (\ref{eq:ugim}) gives $u_H - u_{H'}
\in\mathcal{C}(N)$.
Let $v_H = \sum_{H' \in\mathfrak{T}(H)} u_{H'}$ given in \eqref
{eq:ugim1}. Then $v_H$ is written as
\begin{eqnarray*}
v_H &=& \sum_{H' \in\mathfrak{T}(H)} \biggl\{
\delta_N - \sum_{V
\in\mathcal{V}_{H}}
\delta_V + \sum_{S \in\mathcal{S}_H} \nu
_{H}(S) \cdot\delta_S + \sum_{V \in\mathcal{V}_{H}}
u_{H'_V} \biggr\}
\\
&=& \bigl|\mathfrak{T}(H)\bigr| \cdot \biggl\{ \delta_N - \sum
_{V \in\mathcal
{V}_{H}} \delta_V + \sum
_{S \in\mathcal{S}_H} \nu_{H}(S) \cdot \delta_S \biggr
\} + \sum_{H' \in\mathfrak{T}(H)} \sum_{V \in\mathcal{V}_{H}}
u_{H'_V}
\\
&=& \bigl|\mathfrak{T}(H)\bigr| \cdot \biggl\{ \delta_N - \sum
_{V \in\mathcal
{V}_{H}} \delta_V + \sum
_{S \in\mathcal{S}_H} \nu_{H}(S) \cdot \delta_S \biggr
\}
\\
&&{} + \sum_{V \in\mathcal{V}_{H}} \sum_{G \in\mathfrak{T}(H_V)}
n_H(V,G) \cdot u_G,
\end{eqnarray*}
where $n_H(V,G) = |\{ H' \in\mathfrak{T}(H)\dvt  H'_V = G \}|$ for $V
\in\mathcal{V}_{H}$ and $G \in\mathfrak{T}(H_V)$,
is the number of minimal triangulations $H' \in\mathfrak{T}(H)$ such
that $H'_V = G$.
Note that $u_H$ in \eqref{eq:ugim} is obtained by replacing the
coefficients of the right-hand side by one.
Thus, $u_H$ and $v_H$ belong to the relative interior of the same face
of $\cone(\mathcal{E}(N))$.
Hence, we have $\mathcal{M}_{u_H} = \mathcal{M}_{v_H}$, which means
$\mathcal{M}_{u_H} = \mathcal{M}_H$ from Theorem~\ref{thm:ug}.
\end{pf}
%
\begin{exmp}
Consider the graph $H$ in Figure~\ref{fig:H} again.
The sets of mp-components and clique minimal vertex separators are
$\mathcal{V}_{H} = \{ abcd, cde \}$ and $\mathcal{S}_H = \{ cd \}$.
Since $V_2 = cde$ is a clique, the minimal triangulation of its
subgraph $H_{V_2}$ is itself.
As for $V_1 = abcd$, the minimal triangulations of $H_{V_1}$ are given
in Figure~\ref{fig:new_imset}(b), (c).
Then the standard imset for $H$ in \eqref{eq:ugim} is
\begin{eqnarray*}
u_H &=& \delta_{abcde} - \delta_{abcd} -
\delta_{cde} + \delta_{cd}
\\
& &{}+ (\delta_{abcd} - \delta_{abd} - \delta_{bcd} +
\delta_{bd}) + (\delta_{abcd} - \delta_{abc} -
\delta_{acd} + \delta_{ac})
\\
&=& u_{\langle ab, e \, | \, cd \rangle} +u_{\langle a, c \, | \, bd
\rangle} + u_{\langle b, d \, | \, ac \rangle},
\end{eqnarray*}
which coincides with \eqref{eq:new}.
\end{exmp}

As in the case of directed acyclic graphs and decomposable graphs,
our standard imset for an undirected graph provides a simpler criterion.
%
\begin{cor}
For an undirected graph $H$ and every triplet $\langle A, B \, | \, C
\rangle \in
\mathcal{T}(N)$, the followings are equivalent:
\begin{enumerate}[(iii)]
\item[(i)]$\langle A, B \, | \, C \rangle \in\mathcal{M}_H$,
\item[(ii)]$u_H - u_{\langle A, B \, | \, C \rangle} \in\mathcal{C}(N)$,
\item[(iii)]$u_H - u_{\langle A, B \, | \, C \rangle} \in\mathcal{S}(N)$.
\end{enumerate}
\end{cor}
\begin{pf}
The implication (ii) $\Rightarrow$ (iii) $\Rightarrow$ (i) is obvious
from the definition
and Theorem~\ref{thm:ug_main}.
Thus, we only need to consider the implication (i) $\Rightarrow$ (ii).
For $\langle A, B \, | \, C \rangle \in\mathcal{M}_H$,
Lemma~\ref{lem:ce} implies that $\langle A, B \, | \, C \rangle \in
\mathcal{M}_{H'}$ for
some minimal triangulation $H'$ of $H$.
Hence, $u_{H'} - u_{\langle A, B \, | \, C \rangle} \in\mathcal
{C}(N)$ from Lemma~\ref{lem:criteria}.
For every $V \in\mathcal{V}_{H}$, some minimal triangulation $G$ of
$H_V$ coincides with $H'_V$ from Lemma~\ref{lem:tri-mp}.
Thus, $u_H - u_{H'} \in\mathcal{C}(N)$ from \eqref{eq:u_H'}, which
implies that
\[
u_H - u_{\langle A, B \, | \, C \rangle} = (u_H - u_{H'}) +
(u_{H'} - u_{\langle A, B \, | \, C \rangle}) \in\mathcal{C}(N).
\]
\upqed
\end{pf}
%
\begin{rem}
In the case of directed acyclic graphs and chain graphs, some graphs
may induce
the same conditional independence model,
and we have to consider the uniqueness of standard imsets for these
graphs (cf. Lemma~\ref{lem:dagim}).
However, in the case of undirected graphs, two different graphs cannot
have the same conditional independence model.
Thus, it is not necessary to consider the uniqueness question.
\end{rem}

\section{Standard imsets for general chain graphs}
\label{sec:CG}

In this section, we define a standard imset for a chain graph, which is
a generalization of an undirected graph and a directed acyclic graph.
Studen\'y and Vomlel \cite{stuvom}, and Studen\'y, Roverato and \v
{S}t\v{e}p\'{a}nov\'{a} \cite{sturovste}
give standard imsets for chain graphs which are equivalent to some
directed acyclic graph.
Using this result, we can derive imsets for general chain graphs.
Moreover, we show that these imsets fully represent CG models by
similar arguments as in the case of undirected graphs.
In the later part of this section, we show the uniqueness of these
imsets for %
equivalent chain graphs using the concept of a feasible merging.

\subsection{Generalization of a triangulation to chain graphs}
\label{subsec:cg-triangulation}
First, we introduce a concept which generalizes a triangulation of an
undirected graph.
In the case of an undirected graph, a triangulation of a graph is
defined as a decomposable graph obtained by adding edges to the input graph.
Since decomposable models %
can be interpreted %
as an undirected graph which is %
equivalent to some directed acyclic graph,
we can define a triangulation of a chain graph in the same way.
%
\begin{defn}
\label{def:chain-triangulation}
A chain graph $H' = (V(H'), E(H')), V(H') = V(H)$ is said to be a
triangulation of a chain graph $H $ if $H'$ satisfies that
\begin{enumerate}[(iii)]
\item[(i)]$a \uedge b \in E(H')$ whenever $a \uedge b \in E(H)$,
\item[(ii)]$a \to b \in E(H')$ whenever $a \to b \in E(H)$, and
\item[(iii)]$H'$ is equivalent to some directed acyclic graph $G$, that is,
$\mathcal{M}_{H'} = \mathcal{M}_G$.
\end{enumerate}
A triangulation $H'$ of $H$ is said to be minimal if
there is no triangulation $H''$ of $H$ such that $E(H'')\subset E(H')$ and
$a\rightarrow b \in E(H')$ whenever $a\rightarrow b\in E(H'')$.
\end{defn}

Note that the notion of a minimal triangulation of Definition~\ref
{def:chain-triangulation} is consistent
with  the notion of a minimal triangulation of an undirected graph.
See also Remark~\ref{rem:chain-minimal-triangulation} below.
Hence for a chain graph $H$, we
also denote the set of its minimal triangulations by ${\mathfrak T}(H)$.

The condition (iii) has been characterized by Andersson \textit{et al.} \cite
{andersson} in graphical terms.
For a chain graph $H$ and a chain component $C \in\mathcal{C}_H$,
a closure graph for $C$ is defined as the moral graph $\clo{H}(C) =
(H_{C \cup\pa(C)})^{\mor}$.
%
\begin{prop}[(Andersson \textit{et al.} \cite{andersson})] \label{prop:andersson}
A chain graph is %
equivalent to some directed acyclic graph if and only if
$\clo{H}(C)$ is decomposable for every chain component $C \in\mathcal{C}_H$.
\end{prop}
%
\begin{lem}[(cf. Remark~4.2 in \cite{andersson})]\label{lem:andersson}
For $a \in N$ and $A \subseteq N$, let $\ch_{A}(a) = \ch(a) \cap A$
be the set of all children in $H$ that occur in $A$.
For any chain component $C \in\mathcal{C}_H$, the closure graph $\clo
{H}(C) = (H_{C \cup\pa(C)})^{\mor}$ is
decomposable if and only if:
\begin{enumerate}[(iii)]
\item[(i)]$H_C$ is decomposable,
\item[(ii)] for every $a \in\pa(C)$, and every non-adjacent pair $c,d \in
\ch_{C}(a)$,
we have $c \indep d \, |\linebreak[4]  (\ch_{C}(a) \setminus cd)\,[H_C]$ (in
particular $\ch_{C} (a) \setminus cd \neq \emptyset$), and
\item[(iii)] for every distinct pair $a,b \in\pa(C)$, and every $c\in\ch
_{C}(a) \setminus\ch_{C}(b), d \in\ch_{C}(b) \setminus\ch_{C}(a)$,
we have $\indtri{c}{d}{(\ch_{C}(a) \ch_{C}(b) \setminus cd)}{H_C}$
(in particular, $\ch_{C}(a) \ch_{C}(b) \setminus cd \neq \emptyset$,
and $c,d$ are non-adjacent).
\end{enumerate}
\end{lem}
%
\begin{exmp}
We show in Figure~\ref{fig:notdagcg} the examples of chain graphs which
violate the conditions of Lemma~\ref{lem:andersson}.
These graphs have only one chain component $C$ and its parent set.
In Figure~\ref{fig:notdagcg}(1), the subgraph $H_C$ is not
decomposable. In
Figure~\ref{fig:notdagcg}(2), %
$c$ and $d$ are not separated by $\ch_{C} (a) \setminus cd = \emptyset
$ because of a path $c \uedge b \uedge d$.
In Figure~\ref{fig:notdagcg}(3), %
$c\in\ch_{C}(a) \setminus\ch_{C}(b)$ and $d \in\ch_{C}(b)
\setminus\ch_{C}(a)$ are adjacent.
Thus, $c$ and $d$ are not separated by $\ch_{C} (a) \ch_{C}(b)
\setminus cd = e$.
Their closure graphs are shown in Figure~\ref{fig:notdagcg2}.
These figures show that they are not decomposable, which implies that
the graphs in Figure~\ref{fig:notdagcg} are not
equivalent to any directed acyclic graph from Proposition~\ref{prop:andersson}.
\end{exmp}
%
\begin{figure}

\includegraphics{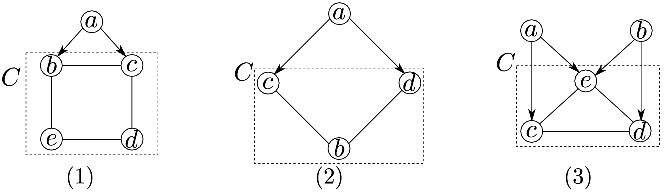}

\caption{Examples of chain graphs violating the conditions of Lemma
\protect\ref{lem:andersson}.}
\label{fig:notdagcg}
\end{figure}

As these facts suggest, it is enough %
to consider a minimal triangulation of $H_{C \cup\pa(C)}$ for each
chain component $C \in\mathcal{C}_H$ instead of the whole $H$.
In fact, if a CG model induced by $H$ coincides with none of DAG
models, then at least one of the conditions (i), (ii) or (iii) in
Lemma~\ref{lem:andersson} is violated.
When these conditions are violated, by adding edges between vertices in
some $C$ or between a vertex in $C$ and a vertex in $\pa(C)$, we can
satisfy these conditions without adding any other edges.
Conversely, all minimal triangulations of $H$ are obtained by a minimal
triangulation of $H_{C \cup\pa_{H}(C)}$ for each chain component $C
\in\mathcal{C}_H$.
Suppose that, for a triangulation $H'$ of $H$, there exists %
$C \in\mathcal{C}_H$ such that the vertex set $V(\clo{H}(C))$ is a
proper subset of $V(\clo{H'}(C'))$ for some $C' \in\mathcal{C}_{H'}$.
This is the case when by the triangulation we add undirected edges
among two distinct components of $H$ or directed edges between the
component $C$ and some vertices which are not parents of $C$ in $H$.
Since $\clo{H'}(C')$ is decomposable by Proposition~\ref
{prop:andersson}, the same
is true for its induced subgraph over the vertex set $V(\clo{H}(C))$.
Thus, we only need to add edges among $C \cup\pa_{H}(C)$.
Moreover,
the set of chain components of a triangulation is identical with that
of the input graph. %

%
\begin{figure}[b]

\includegraphics{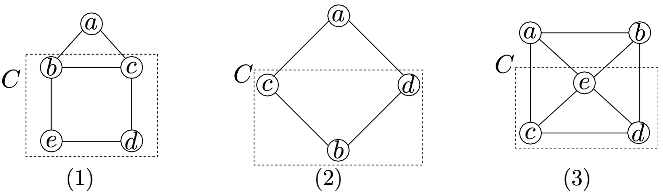}

\caption{The closure graphs of Figure~\protect\ref{fig:notdagcg}.}
\label{fig:notdagcg2}
\end{figure}
%
\begin{rem}
\label{rem:chain-minimal-triangulation}
The above argument shows how to obtain a minimal triangulation of chain graphs.
Let $H$ be a chain graph.
For each $C \in\mathcal{C}_H$ and the parent set
$\pa(C)$ in $H$, a minimal triangulation $H'_{C \cup\pa(C)}$ of
$H_{C \cup\pa(C)}$ is obtained as follows.
Let $G = \clo{H}(C)$ be a closure graph of a chain component $C$ and
$G'$ be a minimal triangulation of $G$.
Then for $F = \clo{E}(G')\setminus\clo{E}(G) $ one constructs a
minimal triangulation $H'_{C \cup\pa(C)}$ by adding
\begin{itemize}
\item an undirected edge $a\uedge b$ provided $(a, b) \in F$ and $a, b
\in C$,
\item a directed edge $a \rightarrow b$ provided $(a, b) \in F$, $a \in
\pa(C)$ and $b \in C$.
\end{itemize}
Indeed, one gets a chain graph consistent with the chain components
order for $H$,
because every vertex in $\pa(C)$ has at least one arrow towards $C$.
%
\begin{figure}

\includegraphics{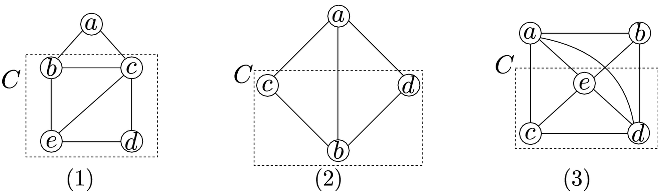}

\caption{Examples of minimal triangulations of Figure~\protect\ref{fig:notdagcg2}.}
\label{fig:cgtriangulation_ug}
\end{figure}
%
\begin{figure}[b]

\includegraphics{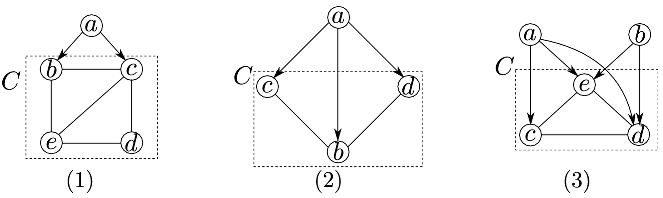}

\caption{Examples of minimal triangulations of  Figure~\protect\ref{fig:notdagcg}.}
        \label{fig:cgtriangulation}
\end{figure}

Note also that in case of
an undirected graph the obtained minimal triangulation is an undirected graph.
\end{rem}
%
\begin{exmp}
Consider minimal triangulations of the graphs in Figure~\ref{fig:notdagcg}.
Examples of minimal triangulations of closure graphs (Figure~\ref{fig:notdagcg2}) for these graphs are shown in Figure~\ref{fig:cgtriangulation_ug}.
In Figure~\ref{fig:cgtriangulation_ug}(1), a minimal triangulation of the
closure graph is obtained by adding the edge $c \uedge e$.
Since $c$ and $e$ belong to the same chain component, adding the edge
$c\uedge e$ gives a minimal triangulation (Figure~\ref{fig:cgtriangulation}(1))
of the chain graph in Figure~\ref{fig:notdagcg}(1).
In Figure~\ref{fig:cgtriangulation_ug}(2), the edge $a\uedge b$ is added.
Since $a$ and $b$ belong to different chain components and there are
directed edges from the chain component of $a$ to that of $b$,
a minimal triangulation (Figure~\ref{fig:cgtriangulation}(2)) of the
graph in
Figure~\ref{fig:notdagcg}(2) is obtained by adding the edge $a \to b$.
As for the conditions of Lemma~\ref{lem:andersson}, $\indtri
{c}{d}{(\ch
_{H,C}(a) \setminus cd)}{H_C}$ holds because $\ch_{H,C}(a) \setminus
cd = b$.
In Figure~\ref{fig:cgtriangulation_ug}(3), we add the edge $a\uedge d$,
hence, obtain the graph in Figure~\ref{fig:cgtriangulation}(3) in the same
way as (2).
Since $\ch_{H,C}(b) \setminus\ch_{H,C}(a) = \emptyset$ in this
graph, the condition (iii) is satisfied automatically.
\end{exmp}

The following lemma immediately holds from the above discussion.
%
\begin{lem} \label{lem:mt1}
For a chain graph $H$ and a chain component $C \in\mathcal{C}_H$,
assume that $\clo{H}(C)$ is decomposable.
Then for every minimal triangulation $H'$ of $H$, we have $H_{C \cup
\pa_H(C)} = H'_{C \cup\pa_{H'}(C)}$.
\end{lem}
%
\begin{cor} \label{cor:mt1}
For a chain graph $H$ and a subset $K \subseteq N$ of the vertex set,
assume that $H_{\an(K)}$ is %
equivalent to some directed acyclic graph.
Then for every minimal triangulation $H'$ of $H$, $H_{\an_{H}(K)} =
H'_{\an_{H'(K)}}$ holds.
\end{cor}
\begin{pf}
Evidently, $\an_{H}(K) = \an_{H'}(K)$ holds. Also, for every chain
component $C \in\mathcal{C}_{H_{\an(K)}}$, the closure graph $\clo
{H}(C)$ is decomposable from Proposition~\ref{prop:andersson}.
Hence, from Lemma~\ref{lem:mt1}, for every minimal triangulation $H'$, we
have $H_{C \cup\pa_H(C)} = H'_{C \cup\pa_{H'}(C)}$,
which implies the corollary.
\end{pf}

As for separations of a chain graph and its minimal triangulation, we
have a similar result to Lemma~\ref{lem:ce} for undirected graphs. See
Figure~\ref{fig:cgconstruction}.
%
\begin{figure}

\includegraphics{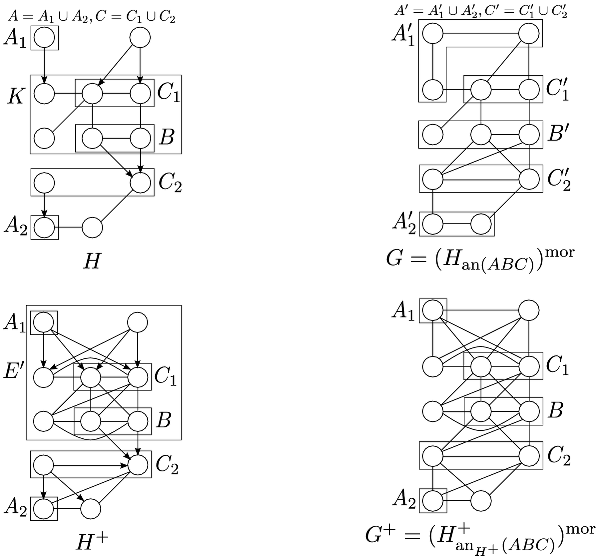}

\caption{A construction of $H^+$ from a chain graph $H$ in Lemma \protect\ref
{lem:cg}.}
\label{fig:cgconstruction}
\end{figure}
%
\begin{lem}\label{lem:cg}
For every chain graph $H$ and triplet $\langle A, B \, | \, C \rangle
\in\mathcal{M}_H$,
there exists a minimal triangulation $H'$ of $H$ such that $\langle A,
B \, | \, C \rangle \in\mathcal{M}_{H'}$.
\end{lem}
\begin{pf}
As in the case of undirected graphs, it suffices to find a
triangulation $H'$ which satisfies $\langle A, B \, | \, C \rangle \in
\mathcal{M}_{H'}$.
First, we construct a triangulation of a subgraph $H_{\an(ABC)}$,
and then consider the whole graph $H$.

From the definition of the separation criterion of a chain graph, we
have $\langle A, B \, | \, C \rangle \in\mathcal{M}_G$ for $G =
(H_{\an(ABC)})^{\mor}$.
We define a partition of $N' = \an_H(ABC)$ as in the same way of the
proof of Lemma~\ref{lem:ce}, that is,
\begin{eqnarray*}
A' &=& \bigl\{ i \in N' \dvt i \mbox{ is connected with }
A \mbox{ in } G_{N'
\setminus C} \bigr\},
\\
C' &=& C \quad \mbox{and}\quad  B' = N' \setminus
A'C'.
\end{eqnarray*}
Then for each $K \in\mathcal{C}_H, K \subseteq N'$, we define a local
graph $H^+(K)$ over
$E' = K \cup\pa_H(K)$ as a graph obtained by removing edges between
$A' \cap E'$ and
$B' \cap E'$ from the graph which has an undirected graph over $K$ and
has all directed edges
$a \to b$ from $a \in\pa(K)$ to $b \in K$.
The graph $H^+$ is defined as the union of these local graphs over $N'$
and outside $H^+_{N'}$ as the same as $H$.
Since $H^+$ is also a chain graph, $H$ and $H^+$ have the same
components and their parent sets.
Therefore, we have $\an_{H^+}(ABC) = \an_H(ABC)$.
Moreover, closure graphs $\clo{H^+}(K)$ for components $K$ are cliques
over $E'$ with removed edges between $A' \cap E'$ and $B' \cap E'$, and
therefore $\clo{H^+}(K)$ is decomposable. This means that $H^+_{\an
_{H^+}(ABC)}$ is equivalent to some acyclic directed graph from
Proposition~\ref{prop:andersson}.
Let $G^+ = (H^+_{\an_{H^+}(ABC)})^{\mor}$.
From the construction of $H^+$, there is no edge between $A'$ and $B'$
in $G^+$.
Thus, we also have $\langle A, B \, | \, C \rangle \in\mathcal{M}_{G^+}$.

Next, we consider the whole graph.
Let $H'$ be a minimal triangulation of $H^+$ (which may be $H^+$ itself).
Since $(H^+_{\an_{H^+}(ABC)})^{\mor}$ is decomposable, $H^+_{\an
_{H^+}(ABC)} = H'_{\an_{H'}(ABC)}$ from Corollary~\ref{cor:mt1}.
Therefore, we have $G' = (H^+_{\an_{H^+}(ABC)})^{\mor} = (H'_{\an
_{H'}(ABC)})^{\mor}$, which implies that $\langle A, B \, | \, C
\rangle \in\mathcal
{M}_{H'}$.
\end{pf}

\subsection{Definition and properties of standard imsets for chain graphs}
In this section, we define a standard imset for a chain graph and show
that it fully represents the CG model induced by this graph.

When a chain graph $H$ is %
equivalent to some directed acyclic graph, its standard imset is
defined as follows \cite{stuvom,sturovste}:\vspace*{2pt}
%
\begin{eqnarray}\label{eq:cgdagim}
u_H = \delta_N - \delta_\emptyset+ \sum
_{C \in\mathcal{C}_H} \biggl\{ \delta_{\pa_H(C)} - \sum
_{K \in\mathcal{K}_{\clo{H}(C)}} \delta_{K} + \sum
_{S \in\mathcal{S}_{\clo{H}(C)}} \nu_{\clo{H}(C)}(S) \cdot \delta_S
\biggr\}.
\end{eqnarray}
This definition is a generalization of that of a directed acyclic graph
\eqref{eq:dagim} and
a decomposable graph \eqref{eq:dgim}. %
Moreover, we have the following lemma about this imset.
%
\begin{prop}[(Studen\'y \textit{et al.} \cite{sturovste})] \label{prop:cgdag_unique}
Assume that two chain graphs $H_1, H_2$ are %
equivalent to some directed acyclic graph.
Then $\mathcal{M}_{H_1} = \mathcal{M}_{H_2}$ if and only if $u_{H_1}
= u_{H_2}$.
\end{prop}
Therefore, for a chain graph $H$ which is equivalent to some directed
acyclic graph, we have $u_H \in\mathcal{C}(N)$ and $\mathcal{M}_H =
\mathcal{M}_{u_H}$ from Lemma~\ref{lem:dagim}.
Furthermore, we have the following corollary from Lemma~\ref{lem:criteria}:
%
\begin{cor} \label{cor:cg_criteria}
Suppose that a chain graph $H$ is equivalent to some directed acyclic
graph and let $u_H$ be given in \eqref{eq:cgdagim}.
For a triplet $\langle A, B \, | \, C \rangle \in\mathcal{T}(N)$,
the followings are
equivalent:
\begin{enumerate}[(iii)]
\item[(i)]$\langle A, B \, | \, C \rangle \in\mathcal{M}_H$,
\item[(ii)]$u_H - u_{\langle A, B \, | \, C \rangle} \in\mathcal{C}(N)$,
\item[(iii)]$u_H - u_{\langle A, B \, | \, C \rangle} \in\mathcal{S}(N)$.
\end{enumerate}
\end{cor}

Note that every closure graph $\clo{H}(C)$, $C \in\mathcal{C}_H$, is
decomposable from Proposition~\ref{prop:andersson}.
Thus \eqref{eq:cgdagim} is also written as\vspace*{2pt}
\begin{eqnarray*}
u_H = \delta_N - \delta_\emptyset+ \sum
_{C \in\mathcal{C}_H} \{ \delta_{\pa(C)} -
\delta_{C \pa(C)} + u_{\clo{H}(C)} \},
\end{eqnarray*}
where $u_{\clo{H}(C)}$ is the standard imset
\eqref{eq:dgim} %
for the decomposable graph $\clo{H}(C)$.
This equation suggests that a generalization of \eqref{eq:cgdagim} is
given by replacing $u_{\clo{H}(C)}$ as in \eqref{eq:ugim}.
For $C \in\mathcal{C}_H$ and $V \subseteq C \cup\pa(C)$,
let $\clo{H}(C)_V$ be the subgraph of the closure graph $\clo{H}(C)$
induced by $V$.
%
\begin{defn} \label{def:cgim}
A standard imset $u_H$ for a chain graph $H$ is defined by\vspace*{2pt}
%
\begin{eqnarray}\label{eq:cgim}
u_H = \delta_N - \delta_\emptyset+ \sum
_{C \in\mathcal{C}_H} \{ \delta_{\pa_H(C)} - \delta_{C \pa_H(C)} +
u_{\clo{H}(C)} \},
\end{eqnarray}
where $u_{\clo{H}(C)}, C \in\mathcal{C}_H$, is the standard imset
for the undirected graph $\clo{H}(C)$ given by \eqref{eq:ugim}:
\begin{eqnarray}
u_{\clo{H}(C)} &=& \delta_{C \pa_H(C)} - \sum_{V \in\mathcal
{V}_{\clo{H}(C)}}
\delta_{V} + \sum_{S \in\mathcal{S}_{\clo
{H}(C)}}
\nu_{\clo{H}(C)}(S) \cdot\delta_S
\nonumber
\\
&&{}+ \sum_{V \in\mathcal{V}_{\clo{H}(C)}} \sum_{G \in\mathfrak
{T}(\clo{H}(C)_V)}
u_{G}.
\nonumber
\end{eqnarray}
\end{defn}
Note that when $H$ is a connected undirected graph this imset coincides
with \eqref{eq:ugim},
because the sum in \eqref{eq:cgim} has only one term and
$\delta_{\pa_H(C)}=\delta_\emptyset$, $\delta_N= \delta_{C \pa_H(C)}$.
We can easily prove that the same conclusion holds for any undirected
graph by considering each connected component.
This imset gives a representation of CG models.
The proof is similar to the case of undirected graphs.
%
\begin{them} \label{thm:cg_imset}
For a chain graph $H$, let a standard imset $u_H$ for $H$ be defined by
\eqref{eq:cgim}.
Then $u_H \in\mathcal{C}(N)$ and $\mathcal{M}_{H} = \mathcal{M}_{u_H}$.
\end{them}
\begin{pf}
The argument in Section~\ref{subsec:cg-triangulation} implies that
$\mathcal{C}_H = \mathcal{C}_{H'}$ and $\pa_H(C) = \pa_{H'}(C),
\forall C \in\mathcal{C}_H$,
for a minimal triangulation $H'$ of $H$.
Thus a standard imset $u_{H'}$ for ${H'}$ given by \eqref{eq:cgdagim} is
\begin{eqnarray*}
u_{H'} &=& \delta_N - \delta_\emptyset+ \sum
_{C \in\mathcal
{C}_{H'}} \{ \delta_{\pa_{H'}(C)} -
\delta_{C \pa_{H'}(C)} + u_{\clod{H}(C)} \}
\\
&=& \delta_N - \delta_\emptyset+ \sum
_{C \in\mathcal{C}_H} \{ \delta_{\pa_{H}(C)} - \delta_{C \pa_{H}(C)} +
u_{\clod{H}(C)} \}.
\end{eqnarray*}
As in the proof (of implication $\mathcal{M}_H \subseteq\mathcal
{M}_{u_H}$) of Theorem~\ref{thm:ug_main}, we have $u_{\clo{H}(C)} -
u_{\clod
{H}(C)} \in\mathcal{S}(N)$ for $C \in\mathcal{C}_H$,
which shows that $u_H - u_{H'} \in\mathcal{S}(N)$.
Also, putting $v_H = \sum_{H' \in\mathfrak{T}(H)} u_{H'}$, we have
\begin{eqnarray*}
v_H &=& \sum_{H' \in\mathfrak{T}(H)} \biggl[
\delta_N - \delta _\emptyset+ \sum
_{C \in\mathcal{C}_H} \{ \delta_{\pa(C)} - \delta_{C \pa(C)} +
u_{\clod{H}(C)} \} \biggr]
\\
&=& \bigl|\mathfrak{T}(H)\bigr| \cdot \biggl[ \delta_N - \delta_\emptyset+
\sum_{C \in\mathcal{C}_H} \{ \delta_{\pa(C)} -
\delta_{C
\pa(C)} \} \biggr] + \sum_{H' \in\mathfrak{T}(H)} \sum
_{C \in\mathcal{C}_H} u_{\clod
{H}(C)}
\\
&= &\bigl|\mathfrak{T}(H)\bigr| \cdot \biggl[ \delta_N - \delta_\emptyset+
\sum_{C \in\mathcal{C}_H} \{ \delta_{\pa(C)} -
\delta_{C
\pa(C)} \} \biggr]
\\
&&{} + \sum_{C \in\mathcal{C}_H} \sum_{G \in\mathfrak{T}(\clo
{H}(C))}
n_H(C, G) \cdot u_{G} ,
\end{eqnarray*}
where $n_H(C,G) = |\{ H' \in\mathfrak{T}(H); \clod{H}(C) = G \}|$
for $C \in\mathcal{C}_H$ and $G \in\mathfrak{T}(\clo{H}(C))$, is
the number of minimal triangulations $H'$ of $H$
such that $\clod{H}(C) = G$.
Therefore, as in the proof of {Theorem~\ref{thm:ug_main} for} the case
of an undirected graph, $u_H$ and $v_H$ belong to
the relative interior of the same face of $\cone(\mathcal{E}(N))$.
Thus, we have $\mathcal{M}_{u_H} = \mathcal{M}_{v_H}$.

For every chain graph,
there exists a discrete measure $\PP$ over $N$ such that $\mathcal
{M}_{\PP} = \mathcal{M}_H$ \cite{stubou}.
Moreover, Theorem~\ref{thm:main} implies that $\mathcal{M}_{\PP} =
\mathcal
{M}_w$ for some $w \in\mathcal{S}(N)$.
Hence for every $H' \in\mathfrak{T}(H)$, we have
\[
\mathcal{M}_{u_{H'}} = \mathcal{M}_{H'} \subseteq
\mathcal{M}_H = \mathcal{M}_{\PP} = \mathcal{M}_w,
\]
which implies that $k_{H'} \cdot w - u_{H'} \in\mathcal{S}(N)$ for
some $k_{H'} \in\NN$ from Lemma~\ref{lem:include}.
Putting $k = \sum_{H' \in\mathfrak{T}(H)} k_{H'}$, we have $k \cdot
w - v_H \in\mathcal{S}(N)$.
Therefore $\mathcal{M}_{u_H} = \mathcal{M}_{v_H} \subseteq\mathcal
{M}_w = \mathcal{M}_H$.

Conversely, for every $\langle A, B \, | \, C \rangle \in\mathcal
{M}_H$, there exists
$H' \in\mathfrak{T}(H)$ such that
$\langle A, B \, | \, C \rangle \in\mathcal{M}_{H'}$ from Lemma~\ref
{lem:cg}.
Thus $u_{H'} - u_{\langle A, B \, | \, C \rangle} \in\mathcal{S}(N)$
from Corollary~\ref{cor:cg_criteria}.
Hence, we have
\[
u_H - u_{\langle A, B \, | \, C \rangle} = (u_H - u_{H'}) +
(u_{H'} - u_{\langle A, B \, | \, C \rangle}) \in\mathcal{S}(N)
\]
and $\langle A, B \, | \, C \rangle \in\mathcal{M}_{u_H}$.
\end{pf}

As in the case of undirected graphs, we have the following corollary.
%
\begin{cor}
For a chain graph $H$ and every triplet $\langle A, B \, | \, C
\rangle \in\mathcal
{T}(N)$, the followings are equivalent:
\begin{enumerate}[(iii)]
\item[(i)]$\langle A, B \, | \, C \rangle \in\mathcal{M}_H$,
\item[(ii)]$u_H - u_{\langle A, B \, | \, C \rangle} \in\mathcal{C}(N)$,
\item[(iii)]$u_H - u_{\langle A, B \, | \, C \rangle} \in\mathcal{S}(N)$.
\end{enumerate}
\end{cor}

\subsection{Feasible merging}
From now on, we will consider the uniqueness of the standard imsets for
chain graphs in Definition~\ref{def:cgim}.

In the case of chain graphs which are equivalent to some directed
acyclic graphs, the uniqueness of their standard imsets defined
by \eqref{eq:cgdagim} is given in Proposition~\ref{prop:cgdag_unique}.
Its proof is based on the concept called a feasible merging \cite{sturovste}.
In this section, we review its definition and properties.

Let $H$ be a chain graph. A pair of its chain components $U,L \in
\mathcal{C}_H$ is said to form a meta-arrow $U \rightrightarrows L$ if
there exists a directed edge $a \to b \in E(H)$ for some $a \in U, b
\in L$.
The \textit{merging of a meta-arrow} $U \rightrightarrows L$ is the
operation of replacing every directed edge $a \to b \in E(H)$, $a \in
U, b \in L$, with $a\uedge b$.
The merging of $U \rightrightarrows L$ is called \textit{feasible} if
the following two conditions are satisfied:
\begin{enumerate}[(ii)]
\item[(i)]$K \equiv\pa(L) \cap U$ is a clique in $H$, and
\item[(ii)]$\pa(L) \setminus U \subseteq\pa(b)$ for any $b \in K$.
\end{enumerate}
By this definition, the merging is feasible if and only if $\pa(L)$ is
a clique in the closure graph $\clo{H}(U)$.
Moreover, for the resulting graph $H'$ and the chain component $M$
obtained by the merging of $U \rightrightarrows L$,
$\pa_H(L)$ is a clique in $\clod{H}(M)$.
%
\begin{figure}

\includegraphics{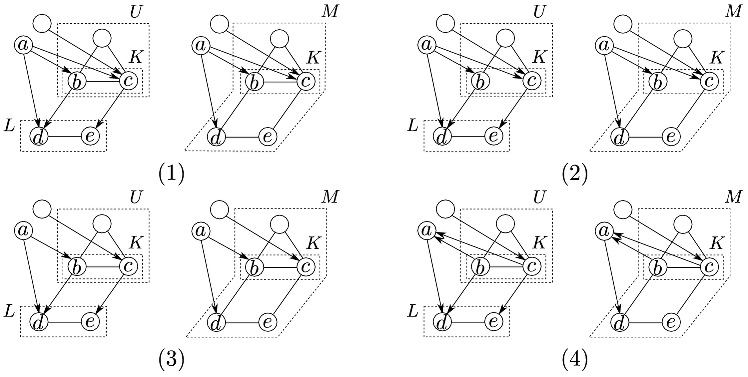}

\caption{Examples of a feasible merging (1) and examples of infeasible
merging (2), (3), (4).}
\label{fig:merging}
\end{figure}
%
\begin{exmp}
We show some examples of feasible and infeasible mergings in Figure~\ref{fig:merging}.
The left-hand side graphs of these figures are input graphs containing
$K = \{ b,c \}, L = \{ d,e \}$ and $\pa(L) = \{ a,b,c \}$,
and the right-hand side graphs the resulting graphs obtained by the
merging $U \rightrightarrows L$ in the input ones.
In Figure~\ref{fig:merging}(1), $K$ is a clique, and $\pa(L) \setminus U = \{ a \} =
\pa(b) \subset\pa(c)$.
Thus both conditions are satisfied, and the merging is feasible.
Especially, the input graph and resulting graph have the same complexes.
In Figure~\ref{fig:merging}(2), since $K$ is not a clique, the condition (i) is not satisfied.
Also in Figure~\ref{fig:merging}(3), the condition (ii) is violated because $\pa(L)
\setminus U = \{ a \} \nsubseteq\pa(c)$.
Hence, the mergings of $U \rightrightarrows L$ in (2) and (3) are infeasible.
Note that, in Figure~\ref{fig:merging}(2), the merging of $U \rightrightarrows L$
destroys a complex $b \to d\uedge e \leftarrow c$.
Similarly, a complex $a \to d\uedge e \leftarrow c$ vanishes in Figure~\ref{fig:merging}(3).
As in Figure~\ref{fig:merging}(3), the condition (ii) is not satisfied in (4).
In this case, the resulting graph has a directed cycle $a \to d\uedge
e\uedge c \to a$,
and hence, it is not a chain graph.
\end{exmp}

As shown in these examples, the resulting graph by a feasible merging
is also a chain graph and has the same complexes as the input graph.
Thus, we have the following important lemma from Theorem~\ref
{thm:cg_equivalent}.
%
\begin{lem}[(Studen\'y \textit{et al.} \cite{sturovste})] %
Let $H$ be a chain graph and $H'$ be a graph obtained by the merging of
$U \rightrightarrows L$ in $H$.
Then $\mathcal{M}_H = \mathcal{M}_{H'}$ if and only if the merging is
feasible.
\end{lem}
The operation of merging can be performed without leaving the %
equivalence class.
Especially, every larger %
equivalent graph is obtained by a series of feasible merging operations.
%
\begin{them}[(Studen\'y \textit{et al.} \cite{sturovste})] \label{thm:merging}%
Let $G$ and $H$ be chain graphs such that $\mathcal{M}_G = \mathcal
{M}_H$ and $H \ge G$.
Then there exists a sequence of chain graphs $G = H_1, \dots, H_r = H,
r \ge1$, such that
$H_{i+1}$ is obtained by the operation of feasible merging in $H_i$ for
all $i = 1, \dots, r-1$.
\end{them}
From Theorem~\ref{thm:frydenberg},
for proving that equivalent chain graphs have a common property,
it suffices to prove that the property is shared by a pair of graphs of
the class
such that one is obtained by a feasible merging from the other.

\subsection{Uniqueness of standard imsets for chain graphs}
In this section, we show that %
equivalent chain graphs have the same standard imset.
%
\begin{them} \label{thm:cg_unique}
Let $H_1, H_2$ be chain graphs.
Then $\mathcal{M}_{H_1} = \mathcal{M}_{H_2}$ if and only if
$u_{H_1} = u_{H_2}$.
\end{them}

To prove this theorem, the following fact is useful.
%
\begin{lem}[(cf. the proof of Theorem~20 in \cite{sturovste})] \label{lem:20}
For a chain graph $H$ which is %
equivalent to some directed acyclic graph,
let $H'$ be a graph obtained from $H$ by a feasible merging of a
meta-arrow $U \rightrightarrows L$,
and let $M$ denote the merged chain component.
Then $K \subseteq N$ is a maximal clique of $ \clod{H}(M)$
if and only if $K$ is either a maximal clique of $ \clo{H}(L)$ or a
maximal clique of $ \clo{H}(U)$
different from $\pa_H(L)$.
\end{lem}

In a chain graph $H$ which is equivalent to some directed acyclic graph,
every mp-subgraph of $\clo{H}(C), C \in\mathcal{C}_H$, is complete,
because a closure graph $\clo{H}(C)$ is decomposable.
As mentioned in Section~\ref{sec:graphs},
the graph obtained by adding edges to an undirected graph such that its all
mp-components become maximal cliques is decomposable.
The following lemma can be easily proved by Lemma~32 of \cite{sturovste}
and Lemma~2.1(i), (ii) of \cite{leimer}.
%
\begin{lem} \label{lem:mp-comp}
For a chain graph $H$, define $H'$ and $M$ as in Lemma~\ref{lem:20}.
Then $K \subseteq N$ is an mp-component of $ \clod{H}(M)$
if and only if $K$ is either an mp-component of $ \clo{H}(L)$ or an
mp-component of $ \clo{H}(U)$
different from $\pa_H(L)$.
\end{lem}

The chain graph in Lemma~\ref{lem:mp-comp} need not be
equivalent to an acyclic directed graph as in Lemma~\ref{lem:20}.
Also note that  $\pa_{H}(L)$ can never be an mp-component of
$\overline H(L)$, and, therefore,
never an mp-component of $\overline{H'}(M)$. This can be shown by
contradiction: if $\pa(L)$
is an mp-component of $\overline H(L)$, then an ordering $V_1, \dots,
V_m$ of its mp-components
satisfying RIP and $V_1 = \pa(L)$ exists
from Theorem~2.5 of \cite{leimer}.
Then $V_1 \setminus S_2$ and $V_2 \setminus S_2$ are separated by $S_2$
from Lemma~\ref{lem:mp-separation-obvious}.
However,
$x \in V_1 \setminus S_2$ must have a child in
$L$, which leads to some vertex in $V_2 \setminus S_2$ since $L$ is a
connected component. This
gives a contradiction with the above separation. %

We now prove Theorem~\ref{thm:cg_unique} using this result.
\begin{pf*}{Proof of Theorem~\ref{thm:cg_unique}}
Let $H = H_1$. Note that the standard imset for $H$ given by \eqref
{eq:cgim} is
\begin{eqnarray}
u_H &=& \delta_N - \delta_\emptyset+ \sum
_{C \in\mathcal{C}_H} \biggl\{ \delta_{\pa(C)} - \sum
_{V \in\mathcal{V}_{\clo{H}(C)}} \delta_{V} + \sum
_{S \in
\mathcal{S}_{\clo{H}(C)}} \nu_{\clo{H}(C)}(S) \cdot\delta_S \biggr
\}
\nonumber
\\
&&{}+ \sum_{C \in\mathcal{C}_H} \sum_{V \in\mathcal{V}_{\clo{H}(C)}}
\sum_{G \in\mathfrak{T}(\clo{H}(C)_V)} u_{G}.
\nonumber
\end{eqnarray}

We first show that $u_H = u_{H'}$ for a chain graph $H'$ obtained from
$H$ by feasible merging of a meta-arrow $U \rightrightarrows L$.
Let $M$ denote the merged chain component.
Since the closure graphs for every chain component $C$ except for
$U,L,M$ are the same in $H$ and in $H'$,
we have to show that
the contribution in $u_{H'}$ corresponding to $M$ is the sum of
contributions in $u_H$ corresponding to $L$ and $U$.
Since $\pa(L)$ is a clique
(in all three considered
graphs $\overline H(L)$, $\overline H(U)$ and $\overline{H'}(M)$),
letting $V = \pa(L)$, we have
\[
\sum_{G \in\mathfrak{T}(\clo{H}(U)_V)} u_G = \sum
_{G \in\mathfrak
{T}(\clo{H}(L)_V)} u_G = \sum_{G \in\mathfrak{T}(\clo{H'}(M)_V)}
u_G = 0.
\]
Also, from Lemma~\ref{lem:mp-comp}, mp-components in $ \clod{H}(M)$ except
for $\pa_H(L)$ are identical with those of either $\clo{H}(L)$ or
$\clo{H}(U)$.
Therefore, we have
%
\begin{eqnarray}\label{eq:1}
\sum_{V \in\mathcal{V}_{\clo{H}(L)}} \sum_{G \in\mathfrak
{T}(\clo{H}(L)_V)}
u_G + \sum_{V \in\mathcal{V}_{\clo{H}(U)}} \sum
_{G \in\mathfrak{T}(\clo{H}(U)_V)} u_G  = \sum_{V \in\mathcal{V}_{\clod{H}(M)}}
\sum_{G \in\mathfrak
{T}(\clod{H}(M)_V)} u_G,
\end{eqnarray}
whether $\pa(L)$ is an mp-component of $\clo{H}(U)$ %
or not.

From \eqref{eq:1} and $\pa_{H}(U) = \pa_{H'}(M)$ (see Lemma~32 in
\cite{sturovste}), $u_H = u_{H'}$ is reduced to
\begin{eqnarray*}
&&- \sum_{V \in\mathcal{V}_{\clo{H}(L)}} \delta_{V} + \sum
_{S \in
\mathcal{S}_{\clo{H}(L)}} \nu_{\clo{H}(L)}(S) \cdot\delta_S +
\delta_{\pa_H(L)}
\\
&&\qquad {}- \sum_{V \in\mathcal{V}_{\clo{H}(U)}} \delta_{V} + \sum
_{S \in
\mathcal{S}_{\clo{H}(U)}} \nu_{\clo{H}(U)}(S) \cdot\delta_S
\\
&&\quad = - \sum_{V \in\mathcal{V}_{\clo{H'}(M)}} \delta_{V} + \sum
_{S
\in\mathcal{S}_{\clo{H'}(M)}} \nu_{\clo{H'}(M)}(S) \cdot
\delta_S,
\end{eqnarray*}
which is the same as the equation (7) in \cite{sturovste} if all
mp-components $V$ are maximal cliques.
Indeed, one can
construct a chain graph $H^*$ over vertices $L\cup U \cup\pa(L)$
having $U$ and $L$ as components
(and possibly some other singleton components in $\pa    (L) \setminus U$)
such that the mp-subgraphs of
$\overline H(L)$ and $\overline H(U)$ are maximal cliques in $\overline
{H^*}(L)$ and $\overline{H^*}(U)$.
This graph is equivalent to an acyclic directed graph, which is the
assumption for validity of
the formula (7) in \cite{stuvom}. %
Then
the similar argument for the proof of Proposition~20 in \cite
{sturovste} holds by Lemma~\ref{lem:mp-comp} and Theorem~2.5 in \cite{leimer}
for the above equation.
Thus we have $u_H = u_{H'}$.

Let $H_\infty$ be the largest chain graph (cf. Theorem~\ref
{thm:frydenberg}) in the equivalence class containing $H_1, H_2$.
Then we have $u_{H_1} = u_{H_\infty}$ from Theorem~\ref{thm:merging}.
Also we have $u_{H_2} = u_{H_\infty}$, which implies Theorem~\ref
{thm:cg_unique}.
\end{pf*}

\section{Concluding remarks}
\label{sec:remarks}
In this paper, we defined standard imsets for undirected graphical
models and
chain graphical models. The crucial concept to derive them was a
minimal triangulation.
For an undirected graph, its imset was defined through all minimal
triangulations of the graph.
Moreover, we gave a more brief form of a standard imset using the
structure of mp-subgraphs.
For a chain graph, we generalized a triangulation of undirected graph.
Then a standard imset for a chain graph was derived through an
analogous argument as the undirected case.
We also showed the uniqueness of standard imsets for equivalent chain graphs.

For directed acyclic graphs and decomposable graphs, the number of
non-zero elements of their standard imsets is linear in $|N|$,
while \eqref{eq:ugim} and \eqref{eq:cgim} may have exponential number
of non-zero elements.
Especially, for a prime undirected graph, imsets defined by \eqref
{eq:ugim1} coincide with \eqref{eq:ugim}.
Thus there is a question whether we can find an imset with smaller
numbers of non-zero elements. %

This is related to the degree of combinatorial imsets. The degree of a
combinatorial imset is defined as
the sum of positive coefficients when it is written as a non-negative
integer combination of elementary imset \cite{stu2005}.
An imset with the smallest degree is considered %
as a basic representative of an equivalence class in Section~7.3 in
\cite{stu2005}.
In fact, a standard imset for a directed acyclic graph has the smallest degree.
Our definition of a standard imset has the smallest degree for some graphs.
One of such examples is a 4-cycle graph. It is easy to see that the
smallest degree in the equivalence class is 2, and
\eqref{eq:ugim} achieves this bound.
Although, for other cycle graphs, \eqref{eq:ugim} does not achieve the
smallest degree,
it may be possible to derive an imset with the smallest degree through
our definition.


\section*{Acknowledgements}
We are very grateful to a referee for very careful and constructive comments.



\printhistory


\begin{thebibliography}{22}


\bibitem{andersson}
\begin{barticle}[mr]
\bauthor{\bsnm{Andersson},~\bfnm{Steen~A.}\binits{S.A.}},
\bauthor{\bsnm{Madigan},~\bfnm{David}\binits{D.}} \AND
\bauthor{\bsnm{Perlman},~\bfnm{Michael~D.}\binits{M.D.}}
(\byear{1997}).
\btitle{On the {M}arkov equivalence of chain graphs, undirected graphs, and acyclic digraphs}.
\bjournal{Scand. J. Stat.}
\bvolume{24}
\bpages{81--102}.
\bid{doi={10.1111/1467-9469.t01-1-00050}, issn={0303-6898}, mr={1436624}}
\end{barticle}
\bptok{imsref}%
\endbibitem

\bibitem{bou2010}
\begin{barticle}[mr]
\bauthor{\bsnm{Bouckaert},~\bfnm{Remco}\binits{R.}},
\bauthor{\bsnm{Hemmecke},~\bfnm{Raymond}\binits{R.}},
\bauthor{\bsnm{Lindner},~\bfnm{Silvia}\binits{S.}} \AND
\bauthor{\bsnm{Studen{\'y}},~\bfnm{Milan}\binits{M.}}
(\byear{2010}).
\btitle{Efficient algorithms for conditional independence inference}.
\bjournal{J. Mach. Learn. Res.}
\bvolume{11}
\bpages{3453--3479}.
\bid{issn={1532-4435}, mr={2756190}}
\end{barticle}
\bptok{imsref}%
\endbibitem

\bibitem{frydenberg}
\begin{barticle}[mr]
\bauthor{\bsnm{Frydenberg},~\bfnm{Morten}\binits{M.}}
(\byear{1990}).
\btitle{The chain graph {M}arkov property}.
\bjournal{Scand. J. Stat.}
\bvolume{17}
\bpages{333--353}.
\bid{issn={0303-6898}, mr={1096723}}
\end{barticle}
\bptok{imsref}%
\endbibitem

\bibitem{geiger1993}
\begin{barticle}[mr]
\bauthor{\bsnm{Geiger},~\bfnm{Dan}\binits{D.}} \AND
\bauthor{\bsnm{Pearl},~\bfnm{Judea}\binits{J.}}
(\byear{1993}).
\btitle{Logical and algorithmic properties of conditional independence and graphical models}.
\bjournal{Ann. Statist.}
\bvolume{21}
\bpages{2001--2021}.
\bid{doi={10.1214/aos/1176349407}, issn={0090-5364}, mr={1245778}}
\end{barticle}
\bptok{imsref}%
\endbibitem

\bibitem{hara}
\begin{barticle}[auto:STB|2014/02/12|14:17:21]
\bauthor{\bsnm{Hara},~\bfnm{H.}\binits{H.}} \AND
\bauthor{\bsnm{Takemura},~\bfnm{A.}\binits{A.}}
(\byear{2010}).
\btitle{A localization approach to improve iterative proportional scaling in Gaussian graphical models}.
\bjournal{Comm. Statist. Theory Methods}
\bvolume{39}
\bpages{1643--1654}.
\end{barticle}
\bptok{imsref}%
\endbibitem

\bibitem{heggernes}
\begin{barticle}[mr]
\bauthor{\bsnm{Heggernes},~\bfnm{Pinar}\binits{P.}}
(\byear{2006}).
\btitle{Minimal triangulations of graphs: A survey}.
\bjournal{Discrete Math.}
\bvolume{306}
\bpages{297--317}.
\bid{doi={10.1016/j.disc.2005.12.003}, issn={0012-365X}, mr={2204109}}
\end{barticle}
\bptok{imsref}%
\endbibitem

\bibitem{hem2010}
\begin{bmisc}[auto:STB|2014/02/12|14:17:21]
\bauthor{\bsnm{Hemmecke},~\bfnm{R.}\binits{R.}},
\bauthor{\bsnm{Lindner},~\bfnm{S.}\binits{S.}} \AND
\bauthor{\bsnm{Studen{\'y}},~\bfnm{M.}\binits{M.}}
(\byear{2010}).
\bhowpublished{Learning restricted Bayesian network structures. Preprint. Available at \arxivurl{arXiv:1011.6664v1}}.
\end{bmisc}
\bptok{imsref}%
%
\endbibitem

\bibitem{hemmecke}
\begin{barticle}[mr]
\bauthor{\bsnm{Hemmecke},~\bfnm{Raymond}\binits{R.}},
\bauthor{\bsnm{Morton},~\bfnm{Jason}\binits{J.}},
\bauthor{\bsnm{Shiu},~\bfnm{Anne}\binits{A.}},
\bauthor{\bsnm{Sturmfels},~\bfnm{Bernd}\binits{B.}} \AND
\bauthor{\bsnm{Wienand},~\bfnm{Oliver}\binits{O.}}
(\byear{2008}).
\btitle{Three counter-examples on semi-graphoids}.
\bjournal{Combin. Probab. Comput.}
\bvolume{17}
\bpages{239--257}.
\bid{doi={10.1017/S0963548307008838}, issn={0963-5483}, mr={2396350}}
\end{barticle}
\bptok{imsref}%
\endbibitem

\bibitem{lauritzen}
\begin{bbook}[mr]
\bauthor{\bsnm{Lauritzen},~\bfnm{Steffen~L.}\binits{S.L.}}
(\byear{1996}).
\btitle{Graphical Models}.
\bseries{Oxford Statistical Science Series}
\bvolume{17}.
\blocation{New York}:
\bpublisher{Oxford Univ. Press}.
\bid{mr={1419991}}
\end{bbook}
\bptok{imsref}%
\endbibitem

\bibitem{laudawlarlei}
\begin{barticle}[mr]
\bauthor{\bsnm{Lauritzen},~\bfnm{S.~L.}\binits{S.L.}},
\bauthor{\bsnm{Dawid},~\bfnm{A.~P.}\binits{A.P.}},
\bauthor{\bsnm{Larsen},~\bfnm{B.~N.}\binits{B.N.}} \AND
\bauthor{\bsnm{Leimer},~\bfnm{H.-G.}\binits{H.-G.}}
(\byear{1990}).
\btitle{Independence properties of directed {M}arkov fields}.
\bjournal{Networks}
\bvolume{20}
\bpages{491--505}.
\bnote{Special issue on influence diagrams}.
\bid{doi={10.1002/net.3230200503}, issn={0028-3045}, mr={1064735}}
\end{barticle}
\bptok{imsref}%
\endbibitem

\bibitem{leimer}
\begin{barticle}[mr]
\bauthor{\bsnm{Leimer},~\bfnm{Hanns-Georg}\binits{H.-G.}}
(\byear{1993}).
\btitle{Optimal decomposition by clique separators}.
\bjournal{Discrete Math.}
\bvolume{113}
\bpages{99--123}.
\bid{doi={10.1016/0012-365X(93)90510-Z}, issn={0012-365X}, mr={1212872}}
\end{barticle}
\bptok{imsref}%
\endbibitem

\bibitem{ohtsuki}
\begin{barticle}[mr]
\bauthor{\bsnm{Ohtsuki},~\bfnm{Tatsuo}\binits{T.}},
\bauthor{\bsnm{Cheung},~\bfnm{Lap~Kit}\binits{L.K.}} \AND
\bauthor{\bsnm{Fujisawa},~\bfnm{Toshio}\binits{T.}}
(\byear{1976}).
\btitle{Minimal triangulation of a graph and optimal pivoting order in a sparse matrix}.
\bjournal{J. Math. Anal. Appl.}
\bvolume{54}
\bpages{622--633}.
\bid{issn={0022-247X}, mr={0485552}}
\end{barticle}
\bptok{imsref}%
\endbibitem

\bibitem{pearl}
\begin{bbook}[mr]
\bauthor{\bsnm{Pearl},~\bfnm{Judea}\binits{J.}}
(\byear{1988}).
\btitle{Probabilistic Reasoning in Intelligent Systems: Networks of Plausible Inference}.
\bseries{The Morgan Kaufmann Series in Representation and Reasoning}.
\blocation{San Mateo, CA}:
\bpublisher{Morgan Kaufmann}.
\bid{mr={0965765}}
\end{bbook}
\bptok{imsref}%
\endbibitem

\bibitem{stu1994}
\begin{barticle}[mr]
\bauthor{\bsnm{Studen{\'y}},~\bfnm{Milan}\binits{M.}}
(\byear{1994/1995}).
\btitle{Description of structures of
stochastic conditional independence by means of faces and imsets (a
series of three papers)}.
\bjournal{Int. J. Gen. Syst.}
\bvolume{23}
\bpages{123--137, 201--219, 323--341}.
\bptnote{check year}%
\end{barticle}
\bptok{imsref}%
\endbibitem


\bibitem{stu2001}
\begin{bincollection}[auto:STB|2014/02/12|14:17:21]
\bauthor{\bsnm{Studen{\'y}},~\bfnm{M.}\binits{M.}}
(\byear{2001}).
\btitle{On non-graphical description of models of conditional independence structure}.
In \bbooktitle{HSSS Workshop on Stochastic Systems for Individual
Behaviours},
\blocation{Louvain la Neuve, Belgium}.
\end{bincollection}
\bptok{imsref}%
\endbibitem




\bibitem{stu2005}
\begin{bbook}[auto:STB|2014/02/12|14:17:21]
\bauthor{\bsnm{Studen{\'y}},~\bfnm{M.}\binits{M.}}
(\byear{2005}).
\btitle{Probabilistic Conditional Independence Structures}.
\blocation{London}:
\bpublisher{Springer}.
\end{bbook}
\bptok{imsref}%
\endbibitem

\bibitem{stubou}
\begin{barticle}[mr]
\bauthor{\bsnm{Studen{\'y}},~\bfnm{Milan}\binits{M.}} \AND
\bauthor{\bsnm{Bouckaert},~\bfnm{Remco~R.}\binits{R.R.}}
(\byear{1998}).
\btitle{On chain graph models for description of conditional independence structures}.
\bjournal{Ann. Statist.}
\bvolume{26}
\bpages{1434--1495}.
\bid{doi={10.1214/aos/1024691250}, issn={0090-5364}, mr={1647685}}
\end{barticle}
\bptok{imsref}%
\endbibitem

\bibitem{sturovste}
\begin{barticle}[mr]
\bauthor{\bsnm{Studen{\'y}},~\bfnm{Milan}\binits{M.}},
\bauthor{\bsnm{Roverato},~\bfnm{Alberto}\binits{A.}} \AND
\bauthor{\bsnm{{\v{S}}t{\v{e}}p{\'a}nov{\'a}},~\bfnm{{\v{S}}{\'a}rka}\binits{\v{S}.}}
(\byear{2009}).
\btitle{Two operations of merging and splitting components in a chain graph}.
\bjournal{Kybernetika (Prague)}
\bvolume{45}
\bpages{208--248}.
\bid{issn={0023-5954}, mr={2518149}}
\end{barticle}
\bptok{imsref}%
\endbibitem

\bibitem{stuvom}
\begin{barticle}[mr]
\bauthor{\bsnm{Studen{\'y}},~\bfnm{Milan}\binits{M.}} \AND
\bauthor{\bsnm{Vomlel},~\bfnm{Ji{\v{r}}{\'{\i}}}\binits{J.}}
(\byear{2009}).
\btitle{A reconstruction algorithm for the essential graph}.
\bjournal{Internat. J. Approx. Reason.}
\bvolume{50}
\bpages{385--413}.
\bid{doi={10.1016/j.ijar.2008.09.001}, issn={0888-613X}, mr={2514506}}
\end{barticle}
\bptok{imsref}%
\endbibitem

\bibitem{stu2010}
\begin{barticle}[mr]
\bauthor{\bsnm{Studen{\'y}},~\bfnm{Milan}\binits{M.}},
\bauthor{\bsnm{Vomlel},~\bfnm{Ji{\v{r}}{\'{\i}}}\binits{J.}} \AND
\bauthor{\bsnm{Hemmecke},~\bfnm{Raymond}\binits{R.}}
(\byear{2010}).
\btitle{A geometric view on learning {B}ayesian network structures}.
\bjournal{Internat. J. Approx. Reason.}
\bvolume{51}
\bpages{573--586}.
\bid{doi={10.1016/j.ijar.2010.01.014}, issn={0888-613X}, mr={2644598}}
\end{barticle}
\bptok{imsref}%
\endbibitem

\bibitem{verma}
\begin{bincollection}[mr]
\bauthor{\bsnm{Verma},~\bfnm{Thomas}\binits{T.}} \AND
\bauthor{\bsnm{Pearl},~\bfnm{Judea}\binits{J.}}
(\byear{1990}).
\btitle{Causal networks: Semantics and expressiveness}.
In \bbooktitle{Uncertainty in Artificial Intelligence, 4}.
\bseries{Mach. Intelligence Pattern Recogn.}
\bvolume{9}
\bpages{69--76}.
\blocation{Amsterdam}:
\bpublisher{North-Holland}.
\bid{mr={1166827}}
\end{bincollection}
\bptok{imsref}%
\endbibitem

\bibitem{vomstu}
\begin{bincollection}[auto:STB|2014/02/12|14:17:21]
\bauthor{\bsnm{Vomlel},~\bfnm{J.}\binits{J.}} \AND
\bauthor{\bsnm{Studen{\'y}},~\bfnm{M.}\binits{M.}}
(\byear{2007}).
\btitle{Graphical and algebraic representatives of conditional independence models}.
In
\bbooktitle{Advances in Probabilistic Graphical Models}
\bpages{55--80}.
\blocation{Berlin}: \bpublisher{Springer}.
\end{bincollection}
\bptok{imsref}%
%
\endbibitem

\end{thebibliography}
\end{document}